\documentclass[11pt,a4paper,twoside]{report}

\usepackage{tabularx}
\usepackage{amssymb}	
\usepackage{amsthm}		
\usepackage{amsmath}	
\usepackage{graphicx}	
\usepackage{multicol} 
\usepackage{enumerate} 
\usepackage{enumitem}
\usepackage{chronology}
\usepackage{graphicx}	

\usepackage[margin=35mm]{geometry}		
\newenvironment{my_enumerate}{
\begin{enumerate}
	\vspace{-1mm}
  \setlength{\itemsep}{3pt}
}{\end{enumerate}
}

\renewcommand{\geq}{\geqslant}		
\renewcommand{\leq}{\leqslant}		
\renewcommand{\mod}{\text{ mod }}	
\newcommand{\fq}{$\mathbb{F}_{q}$}	
\newcommand{\mfq}{\mathbb{F}_{q}}	
\newcommand{\fqs}{$\mathbb{F}_{q}$\space}		
\newcommand{\fqx}{$\mathbb{F}_{q}[x]$}		
\newcommand{\mfqx}{\mathbb{F}_{q}[x]}		
\newcommand{\PP}{{\bf PP}}	

\newtheorem{theorem}{Theorem}[chapter]			
\newtheorem{corollary}[theorem]{Corollary}		
\newtheorem{lemma}[theorem]{Lemma}			
\newtheorem{prop}[theorem]{Proposition}			
\newtheorem{conj}[theorem]{Conjecture}			

\theoremstyle{definition}		
\newtheorem{newdef}{Definition}[chapter]
\newtheorem{example}{Example}[chapter]
\newtheorem{remark}{Remark}[chapter]

\def\eref#1{$(\ref{#1})$}				
\def\lref#1{Lemma~$\ref{#1}$}		
\def\dref#1{Definition~$\ref{#1}$}		
\def\tref#1{Theorem~$\ref{#1}$}		
\def\cref#1{Corollary~$\ref{#1}$}		
\def\chref#1{Chapter~$\ref{#1}$}	
\def\pref#1{Proposition~$\ref{#1}$}

\numberwithin{equation}{chapter} 		
\numberwithin{figure}{chapter}		

\begin{document}

\begin{titlepage}
\begin{center}
\textsc{\huge Monash University}\\[1.5cm]
{ \Huge \bfseries Permutation Polynomials of Finite Fields}\\[1.5cm]
\textsc{\LARGE Honours Project}\\[1.2cm]


\vfill

\begin{minipage}{0.4\textwidth}
\begin{flushleft} \Large
\emph{Author:}\\
Christopher J. Shallue
\end{flushleft}
\end{minipage}
\begin{minipage}{0.4\textwidth}
\begin{flushright} \Large
\emph{Supervisor:} \\
A/Prof. Ian M. Wanless
\end{flushright}
\end{minipage}
\\[2cm]

\textsc{\LARGE May 2012}

\end{center}
\end{titlepage}

$\mbox{ \space}$\\[4cm]
\begin{centering}
{\bf Abstract}\\
\end{centering}
Let \fq\ be the finite field of $q$ elements. Then a \emph{permutation polynomial} (\PP) of \fq\ is a polynomial $f \in \mfqx$ such that the associated function $c \mapsto f(c)$ is a permutation of the elements of \fq. In 1897 Dickson gave what he claimed to be a complete list of \PP s of degree at most 6, however there have been suggestions recently that this classification might be incomplete. Unfortunately, Dickson's claim of a full characterisation is not easily verified because his published proof is difficult to follow. This is mainly due to antiquated terminology. In this project we present a full reconstruction of the classification of degree 6 \PP s, which combined with a recent paper by Li \emph{et al.} finally puts to rest the characterisation problem of \PP s of degree up to 6.

In addition, we give a survey of the major results on \PP s since Dickson's 1897 paper. Particular emphasis is placed on the proof of the so-called \emph{Carlitz Conjecture}, which states that if $q$ is odd and `large' and $n$ is even then there are no \PP s of degree $n$. This important result was resolved in the affirmative by research spanning three decades. A generalisation of Carlitz's conjecture due to Mullen proposes that if $q$ is odd and `large' and $n$ is even then no polynomial of degree $n$ is `close' to being a \PP. This has remained an unresolved problem in published literature. We provide a counterexample to Mullen's conjecture, and also point out how recent results imply a more general version of this statement (provided one increases what is meant by $q$ being `large').
\vfill

\tableofcontents

\chapter{Permutation Polynomials of Finite Fields}
This chapter is devoted to a preliminary exploration of permutation polynomials and a survey of fundamental results. Most of the ideas, results and proofs presented are based on published works of more than century's worth of academic interest in this area. In particular, the reader may find many of the theorems and proofs from this chapter in the excellent treatise on finite fields by Lidl and Neiderreiter \cite[Ch.~7]{lidl}. Some of the omitted proofs can also be found there. We would like to thank A. B. Evans for providing us with a preprint of his book \cite{evans2012}, from which we have used the formula \eref{eq:transposition} and the proof of \tref{thm:quadratic_binomial}. Other published works have been referenced where necessary.

\section{Functions as Polynomials}
Let $q=p^r$, where $p$ is a prime and $r\geq 1$ is an integer. In this project we are interested in functions from the finite field \fq\ into itself, namely functions of the form
\[ \Phi: \mfq \longrightarrow \mfq. \]
To study such functions it is enough to study polynomials of degree at most $q-1$, as the next lemma shows. This result was proved by Leonard Eugene Dickson in 1897 \cite{dickson1897II}; for $q$ prime it was already noted by Hermite \cite{hermite1854}.
\begin{lemma}\label{lem:unique_poly}
For any function $\Phi : \mfq \rightarrow \mfq$ there exists a unique polynomial $f \in\mfqx$ of degree at most $q-1$ such that the associated polynomial function $f: c \mapsto f(c)$ satisfies $\Phi(c) = f(c)$ for all $c \in \mfq.$
\end{lemma}

\begin{proof}
The following formula (\emph{Carlitz Interpolation Formula}) gives a suitable polynomial:
\begin{equation}\label{eq:carlitz}
 f(x) = \sum_{c \in \mfq} \Phi(c) \left( 1-(x-c)^{q-1} \right).
\end{equation}
To show uniqueness, suppose that $f,g \in \mfqx$ are polynomials of degree $\leq q-1$ satisfying $f(c)=g(c)$ for all $c \in \mfq$. If $f \neq g$ then it follows that their difference $f-g$ is a nonzero polynomial that vanishes at all $q$ elements of \fq. But $\deg(f-g) \leq q-1$, so $f-g$ can have at most $q-1$ roots in \fq, a contradiction.
\end{proof}
Note that this lemma establishes a one-to-one correspondence between functions $\Phi : \mfq \to \mfq$ and polynomials $f \in\mfqx$ of degree $\leq q-1$; for there are $q^q$ possible functions each represented uniquely by one of $q^q$ polynomials.

Suppose that $g \in \mfqx$ is a polynomial with degree exceeding $q-1$. Using \eref{eq:carlitz} we can find the unique polynomial $f$ of degree $\leq q-1$ that induces the same function on the underlying field. The following lemma shows we can also find $f$  by reduction modulo $x^q-x$.
\begin{lemma}\label{lem:reduction}
For any $f,g \in \mfqx$ we have $f(c)=g(c)$ for all $c\in \mfq$ if and only if $f(x) \equiv g(x) \mod{(x^q-x)}$.
\end{lemma}

\begin{proof}
By the division algorithm we can write
\[ f(x)-g(x)=h(x)(x^q-x)+r(x), \text{ where } \deg(r) <q. \]
Then $f(c)-g(c)=r(c)$ for all $c \in \mfq$, so $f(c)=g(c)$ for all $c \in \mfq$ if and only if $r$ vanishes at every element of \fq. Since $\deg(r) <q$ this is equivalent to $r(x)=0$.
\end{proof}

\section{Permutation Polynomials}
More specifically, the objects of interest in this project are functions $f : \mfq \rightarrow \mfq$ that permute the elements of \fq. That is, we are interested in bijections of \fq. By \lref{lem:unique_poly} we may assume that such a function is a polynomial of degree at most $q-1$.

\begin{newdef}\label{def:pp}
A polynomial $f \in \mfqx$ is called a {\bf Permutation Polynomial} (\PP) of \fq\ if the associated polynomial function $f: c \rightarrow f(c)$ is a permutation of \fq.
\end{newdef}
By the finiteness of \fq\ we can express this definition in several equivalent ways.
\begin{lemma}\label{lem:alt_defns}
The polynomial $f \in \mfqx$ is a permutation polynomial of \fq\ if and only if one of the following conditions holds:
\begin{my_enumerate}
\item the function $f: c \mapsto f(c)$ is one-to-one;
\item the function $f: c \mapsto f(c)$  is onto;
\item $f(x)=a$ has a solution in \fq\ for each $a \in \mfq$;
\item $f(x)=a$ has a unique solution in \fq\ for each $a \in \mfq$.
\end{my_enumerate}
\end{lemma}

\begin{example}\label{eg:pp_eg1}
Consider the polynomial 
\begin{align*}
f(x) &= 3 x^9 + 7 x^8 +4 x^7 + 9 x^6 + 8 x^5 + 6 x^4 + 2 x^3 + 5 x^2 +x+1  \\
	&= 3 (x+9) (x^4 +5x+8) (x^4+ 8 x^3 + 10 x^2+ 7 x +8) \in \mathbb{F}_{11}[x].
\end{align*}
By computing its values on the set $\{ 0,1,...,10 \} = \mathbb{F}_{11}$ we have
\[
 \begin{array}{c|c|c|c|c|c|c|c|c|c|c|c}
  x &0&1&2&3&4&5&6&7&8&9&10 \\
  \hline
	f(x)&1&2&0&3&4&5&6&7&8&9&10
 \end{array}.
\]
Since $f(x)$ is a bijection it is a permutation polynomial of $\mathbb{F}_{11}$, and we observe that it represents the 3-cycle $(0,1,2)$. \qed
\end{example}

\begin{example}\label{eg:pp_eg2}
Consider the polynomial 
\[ g(x)= x^3+1 \in \mathbb{F}_{11}[x]. \]
As in the previous example we check whether $g$ is a \PP\ of $\mathbb{F}_{11}$ by computing its values on $\mathbb{F}_{11}$. We get
\[
 \begin{array}{c|c|c|c|c|c|c|c|c|c|c|c}
  x &0&1&2&3&4&5&6&7&8&9&10 \\
  \hline
	g(x)&1&2&9&6&10&5&8&3&7&4&0
 \end{array}.
\]
We see that $g$ is a \PP\ of $\mathbb{F}_{11}$ with cycle structure $(0,1,2,9,4,10)(3,6,8,7)$. \qed
\end{example}

\begin{example}\label{eg:pp_eg3}
Finally, consider the polynomial 
\[ h(x)= x^2 +3x+5 \in \mathbb{F}_{11}[x], \]
which takes the values
\[
 \begin{array}{c|c|c|c|c|c|c|c|c|c|c|c}
  x &0&1&2&3&4&5&6&7&8&9&10 \\
  \hline
	h(x)&5&9&4&1&0&1&4&9&5&3&3
 \end{array}.
\]
We see that $h(x)$ is \emph{not} a \PP\ of $\mathbb{F}_{11}$. This is also clear if we write $h$ in the form 
\[ h(x)= (x+7)^2, \]
and observe that since $x^2$ is not an onto function neither is any function composed with $x^2$. \qed
\end{example}

\begin{remark}
Examples \ref{eg:pp_eg1} and \ref{eg:pp_eg2} demonstrate a noteworthy fact on the relationship between permutations and their associated polynomials: simplicity of cycle structure does not imply simplicity as a polynomial, and vice versa. In fact, let $a,b \in \mfq$ and consider the transposition $(a,b)$; the permutation with simplest nontrivial cycle structure. By \eref{eq:carlitz} we determine that the \PP\ representing $(a,b)$ is given by
\begin{equation}\label{eq:transposition}
f(x)=x+(b-a)(1-(x-a)^{q-1})+(a-b)(1-(x-b)^{q-1}).
\end{equation}
Clearly, this is a more complex structure than its cycle form. 

In fact, it is true in general that permutations with simple cycle structure tend to have complex polynomial structure. The interested reader may refer to \cite{wells1969}, which shows that most permutations that move very few elements have maximum possible degree. For example, all transpositions and almost all 3-cycles have maximal degree.
\end{remark}

\section{Criteria for Permutation Polynomials}
Given a polynomial $f\in \mfqx$ it is natural to ask: is $f(x)$ a \PP\ of \fq? For an arbitrary polynomial $f$ this is a difficult question to answer. A straightforward approach (as in Examples \ref{eg:pp_eg1} - \ref{eg:pp_eg3}) is to evaluate $f(c)$ for each $c \in \mfq$, and determine by examination whether or not $f$ is a bijection. If $q$ and $\deg(f)$ are small this is plausible, however in general it is computationally impractical. Although there do exist other techniques, all currently known criteria for \PP s are complicated by way of requiring long calculations. There are no methods that allow an arbitrary polynomial to be checked by inspection, for example. 

In this section we aim to give a fairly comprehensive survey of all known criteria for \PP s. First we give considerable attention to a classical result known as Hermite's criterion. This theorem was first given by Hermite for fields of prime order \cite{hermite1854}, and was later generalised by Dickson to general finite fields \cite{dickson1897II}. We will use this theorem extensively in \chref{chap:deg6}.

\subsection{Hermite's Criterion}
Permutation polynomials of \fq\ may be characterised as polynomial functions $f \in \mfqx$ satisfying the property
\[ \{ f(c) : c \in \mfq \} = \mfq. \]
Hence, it is useful to have a characterisation of sequences $a_0,a_1,...,a_{q-1}$ of elements of \fq\ that satisfy $\{ a_0,a_1,...,a_{q-1} \} = \mfq.$ Note that the set $\{ f(c) : c \in \mfq \}$ is known as the \emph{value set of f}, denoted $V_f$, and will be discussed further in \chref{chap:carlitz}. 

For the following lemma we must first recall the formula for the sum of the first $n$ terms of a geometric series. Let $F$ be a field and let $a \in F$, $a \neq 1$. Then the following identity holds
\begin{equation}\label{eq:geometric}
\sum_{i=0}^{n-1} a^i = \frac{(1-a^n)}{1-a}.
\end{equation}

\begin{lemma}\label{lem:before_hermite}
The sequence $ a_0,...,a_{q-1}$ of elements of \fq\ satisfies $\{ a_0,...,a_{q-1} \} = \mfq$ if and only if
\[ 
\sum_{i=0}^{q-1} a_i^t = 
\begin{cases}
0 &\text{for } t=0,1,...,q-2,\\
-1 &\text{for } t=q-1.
\end{cases}
\]
\end{lemma}

\begin{proof}
For each $0 \leq i \leq q-1$ consider the polynomial
\[ g_i(x)=1-\sum_{t=0}^{q-1} a_i^{t} x^{q-1-t}. \]
It is clear that $g_i(a_i)=1$ for all $0 \leq i \leq q-1$. Note that we also have $g_i(b)=0$ for all $b \in \mfq, b \neq a_i$. To show this, suppose that $b \neq 0$. Then by \eref{eq:geometric} we have
\[ g_i(b)=1-\sum_{t=0}^{q-1} a_i^{t} b^{q-1-t} =1-\sum_{t=0}^{q-1} (a_ib^{-1})^t = 1 - \frac{1- (a_i b^{-1})^q}{1- (a_ib^{-1})} = 1- 1 =0.   \]
Moreover it is clear that $g_i(0)=0$ whenever $a_i \neq 0$. Hence the polynomial
\begin{equation}\label{eq:g_proving_hermite}
g(x) = \sum_{i=0}^{q-1} g_i(x) = - \sum_{i=0}^{q-1} \left( \sum_{t=0}^{q-1} a_i^{t} x^{q-1-t} \right)=  - \sum_{t=0}^{q-1} \left( \sum_{i=0}^{q-1} a_i^{t} \right) x^{q-1-t}
\end{equation}
satisfies
\[ g(x) = 
\begin{cases}
1 &\text{if } x \in \{ a_0,a_1,...,a_{q-1} \}, \\
0 &\text{if } x \in \mfq \setminus \{ a_0,a_1,...,a_{q-1} \}. \\
\end{cases}
\]
So $g(x)$ maps every element of \fq\ to 1 if and only if $\{ a_0,a_1,...,a_{q-1} \} = \mfq$. But since $\deg(g) \leq q-1$ we have by \lref{lem:unique_poly} that $g$ maps every element to 1 if and only if $g(x)=1$, which by \eref{eq:g_proving_hermite} is equivalent to
\[ 
\sum_{i=0}^{q-1} a_i^t = 
\begin{cases}
0 &\text{for } t=0,1,...,q-2,\\
-1 &\text{for } t=q-1.
\end{cases}
\]
\end{proof}

The following criterion for permutation polynomials is known as Hermite's criterion.
\begin{theorem}{\emph{\bf (Hermite's criterion.)}}\label{thm:hermite}
Let $q=p^r$, where $p$ is a prime and $r$ is a positive integer. Then a polynomial $f \in \mfqx$ is a \PP\ of \fq\ if and only if the following two conditions hold:
\begin{my_enumerate}
\item the reduction of $f(x)^{q-1} \mod (x^q-x)$ is monic of degree $q-1$;
\item for each integer $t$ with $1 \leq t \leq q-2$ and $t \not\equiv 0 \mod p$, the reduction of $f(x)^t \mod (x^q-x)$ has degree $\leq q-2$.
\end{my_enumerate}
\end{theorem}

\begin{proof}
For each $1 \leq t \leq q-1$, denote the reduction of $f(x)^t$ modulo $x^q-x$ by
\[ f(x)^t \mod (x^q-x) = \sum_{i=0}^{q-1} b_i^{(t)} x^i. \]
Note that by \eref{eq:carlitz} we have $b_{q-1}^{(t)}=-\sum_{c \in \mfq} f(c)^t$. 

Suppose that $f(x)$ is a \PP\ of \fq. Then since $\{ f(c) : c \in \mfq \} = \mfq$ we have by \lref{lem:before_hermite} that $b_{q-1}^{(t)}=0$ for all $1 \leq t \leq q-2$ and $b_{q-1}^{(q-1)}=1$. 

Now suppose that $(1)$ and $(2)$ are satisfied. Then $(1)$ implies that $-b_{q-1}^{(q-1)}=\sum_{c \in \mfq} f(c)^{q-1} =-1$, whilst $(2)$ implies that $-b_{q-1}^{(t)}=\sum_{c \in \mfq} f(c)^{t} = 0$ for all $1 \leq t \leq q-2$, $t \not\equiv 0 \mod p$. If $t \equiv 0 \mod p$ we may write $t=t' p^j$, where $1 \leq t' \leq q-2$ and $t' \not\equiv 0 \mod p$. We then have
\[ \sum_{c \in \mfq} f(c)^{t} = \sum_{c \in \mfq} f(c)^{t' p^j} = \left( \sum_{c \in \mfq} f(c)^{t'} \right) ^{p^j} =0. \]
So $\sum_{c \in \mfq} f(c)^{t} = 0$ for all $1 \leq t \leq q-2$ and this identity also holds trivially for $t=0$. By \lref{lem:before_hermite}, $f(x)$ is a \PP\ of \fq. 
\end{proof}

In the previous proof it is possible to remove the condition that the reduced polynomial in (1) is monic; it is enough to say that its degree is $q-1$. Alternatively, we can replace condition (1) in \tref{thm:hermite} by other conditions. The following theorem is an equivalent form of Hermite's criterion, and in fact is very close to the original statement proved by Dickson in 1897.

\begin{theorem}\label{thm:hermite2}
Let $q=p^r$, where $p$ is a prime and $r$ is a positive integer. Then a polynomial $f \in \mfqx$ is a \PP\ of \fq\ if and only if the following two conditions hold:
\begin{my_enumerate}
\item $f$ has exactly one root in \fq;
\item for each integer $t$ with $1 \leq t \leq q-2$ and $t \not\equiv 0 \mod p$, the reduction of $f(x)^t \mod (x^q-x)$ has degree $\leq q-2$.
\end{my_enumerate}
\end{theorem}

\begin{proof}
We wish to prove that $f$ has exactly one root in \fq\ if and only if the reduction of $f(x)^{q-1} \mod (x^q-x)$ is monic of degree $q-1$.
As in the proof of \tref{thm:hermite} we write
\[ f(x)^t \mod (x^q-x) = \sum_{i=0}^{q-1} b_i^{(t)} x^i, \]
where $b_{q-1}^{(t)}=-\sum_{c \in \mfq} f(c)^t$. Suppose that $f$ has exactly $j$ roots in \fq. Then 
\[ b_{q-1}^{(q-1)}=-\sum_{c \in \mfq} f(c)^{q-1}=-(q-j)=j, \]
and since $0 \leq j \leq q-1$ we have $b_{q-1}^{(q-1)}=1$ if and only if $j=1$.
\end{proof}

Hermite's criterion gives us some immediate and very useful corollaries. We first show that every reduced \PP\ of \fq\ must have degree $\leq q-2$.

\begin{corollary}\label{cor:max_degree}
If $q>2$ and $f(x)$ is a \PP\ of \fq\ then the reduction of $f$ modulo $x^q-x$ has degree at most $q-2$.
\end{corollary}

\begin{proof}
Set $t=1$ in \tref{thm:hermite}.
\end{proof}

\begin{corollary}\label{cor:notequiv1}
If $q \equiv 1 \mod n$ then there is no \PP\ of \fq\ of degree $n$.
\end{corollary}

\begin{proof}
Let $f(x) \in \mfqx$, where $q =p^r= nm+1$ for some positive integer $m$. By \lref{lem:reduction} we may assume that $n \leq q-1$. Then $1 \leq m \leq q-1$ for all $n \geq 1$, and $m \not\equiv 0 \mod p$ (otherwise $0 \equiv 1 \mod p$). But $\deg(f(x)^{m})=nm=q-1$, so by \tref{thm:hermite} $f(x)$ is not a \PP\ of \fq. 
\end{proof}

\subsection{Survey of Known Criteria}
Recall that a \emph{character} $\chi$ of a finite abelian group $G$ is a homomorphism from $G$ into the multiplicative group $U$ of complex numbers of unit absolute value. The number of characters of $G$ is equal to $|G|$. If \fq\ is a finite field then an \emph{additive character} of \fq\ is a character of the additive group of \fq, that is, a function $\chi : \mfq \to U$ such that 
\[ \chi(x_1 + x_2) = \chi(x_1) \chi(x_2) \text{ for all } x_1,x_2 \in \mfq. \]
The \emph{trivial additive character} $\chi_0$ of \fq\ is defined by $\chi_0 (c)=1$ for all $c \in \mfq$; all other additive characters are considered nontrivial.

The following characterisation of \PP s of \fq\ is well known, see for example \cite{lidl}.
\begin{theorem}
A polynomial $f\in \mfqx$ is a \PP\ of \fq\ if and only if 
\[ \sum_{c \in \mfq} \chi(f(c))=0\]
for all nontrivial additive characters $\chi$ of \fq.
\end{theorem}
The following characterisation of \PP s dates back to 1883 and is due to Raussnitz. The version given here is from \cite{turnwald1995}, where the reader may also find its proof. The same theorem can also be found in \cite[p.~133]{mullen1991}. We have included a reference to the original paper of Raussnitz \cite{raussnitz1883}, however we remark that we were not able to find a copy.

Recall that the \emph{circulant matrix} with first row $(a_0,...,a_n)$ is defined by
\[ M = 
\begin{pmatrix}
a_0 & a_1 & \cdots & a_{n} \\
a_{n} & a_0 & \cdots & a_{n-1} \\
\vdots &\vdots & \ddots & \vdots \\
a_1 & a_2&\cdots &a_0
\end{pmatrix}.
\]

\begin{theorem}{\emph{\bf (Raussnitz).}}
Consider the polynomial $f(x) = \sum_{i=0}^{q-2} a_i x^i$ and let $M_f$ be the circulant matrix with first row $(a_0,a_1,...,a_{q-2})$. Then $f(x)$ is a \PP\ of \fq\ if and only if the characteristic polynomial of $M_f$ is $(x-a_0)^{q-1}-1$.
\end{theorem}
In \cite{gathen1991a} the author derives the following criterion equivalent to the theorem of Raussnitz.  If $f(x) = \sum_{i=0}^n a_i x^i$ and $g(x) = \sum_{i=0}^m b_i x^i$, define the \emph{Sylvester matrix} of $f$ and $g$ by
\[ R(f,g) = 
\begin{pmatrix}
a_n&a_{n-1}& \cdots & a_0 &  & &\\
 & a_n&a_{n-1}& \cdots & a_0 & &\\
&& \ddots &\ddots &  &\ddots & \\
&&&a_n&a_{n-1}& \cdots & a_0 \\
b_m&b_{m-1}& \cdots & b_0 &  &&\\
 & b_m&b_{m-1}& \cdots & b_0 &&\\
&& \ddots &\ddots & &\ddots & \\
&&&b_m&b_{m-1}& \cdots & b_0
\end{pmatrix}
.\]

\begin{theorem}\label{thm:raussnitz2}
Let $f \in \mfqx$ and let
\[ g_f= \det \left ( R(x^q-x,f-y) \right) -(-1)^q (y^q-y) \in \mathbb{F}_q[y]. \]
Then $f(x)$ is a \PP\ of \fq\ if and only if $g_f=0$.
\end{theorem}

By studying elementary symmetric polynomials, Turnwald \cite[Theorem 2.13]{turnwald1995} proves a theorem giving no less than nine characterisations of \PP s. Let $f \in \mfqx$ be a polynomial of degree $n$ such that $1 \leq n <q$ and let $s_k$ be the $k^{th}$ elementary symmetric polynomial of the values $f(c)$, that is, 
\begin{equation}
\prod_{c \in \mfq} (x-f(c)) = \sum^q_{k=0} (-1)^k s_k x^{q-k}.
\end{equation}
Let $u$ be the smallest positive integer $k$ such that $s_k=0$ and let $w$ be the smallest positive integer $k$ such that $p_k=\sum_{c \in \mfq} f(c)^k \neq 0$. Let $v$ be the number of distinct values of $f$. In \cite{turnwald1995} the author studies the relationships between the values $u,v,w,n$ and $q$, in particular deriving the following characterisations of the statement $v=q$ (i.e. $f$ is a \PP ).

\begin{theorem}\label{thm:9criteria}
Let $f \in \mfqx$ be a polynomial of degree $n$ with $1 \leq n <q$ and let $u,w,v$ be as defined above. Then the following statements are equivalent:
\vspace{-2mm}
\begin{multicols}{2}
\begin{my_enumerate}
\item $f(x)$ is a \PP.
\item $u=q-1$.
\item $u > q-q/n$.
\item $u>q-v$.
\item $v > q- (q-1)/n$.
\item $w=q-1$.
\item $2q/3 -1 < w < \infty$.
\item $q-(q+1)/n < w < \infty$.
\item $q-u \leq w < \infty$.
\item $u > (q-1)/2$ and $w < \infty$.
\end{my_enumerate}
\end{multicols}
\end{theorem}
The remarkable fact that $v > q- (q-1)/n$ implies $v=q$ is a theorem due to Wan, which we will discuss further in \chref{chap:carlitz}.

For completeness of this survey we give a final criterion that has been reported in the literature. According to a statement in \cite[p.~251]{mullen1995}, the following theorem is taken from a preprint of Moreno \emph{et al.}, however it does not seem that the paper in question was published. The reference of this preprint may be found in the bibliography of \cite{mullen1995}.
\begin{theorem}
A polynomial $f \in \mfqx$ is a \PP\ of \fq\ if and only if one of the following conditions holds:
\begin{my_enumerate}
\item $(f(x)-c)^{q-1} \not \equiv 1 \mod (x^q-x)$ for all $c \in \mfq$.
\item $(f(x)-f(c))^{q-1} \equiv (x-c)^{q-1} \mod (x^q-x)$ for all $c \in \mfq$.
\end{my_enumerate}
\end{theorem}

At the conclusion of this section we remark that all the criteria listed here are computationally demanding, even for polynomials of small degrees over small fields. For this reason, some of the above criteria have been converted into probabilistic algorithms for testing for \PP s. In particular, the reader is referred to \cite{gathen1991a} for probabilistic versions of \tref{thm:hermite2} and \tref{thm:raussnitz2}. See also \cite{gathen1991,shparlinski1992}.

\section{Classes of Permutation Polynomials}
We have seen that in general it is difficult to tell whether or not an arbitrary polynomial is a \PP. However, for certain special classes of polynomials this question is easier to answer. In this section we give a survey of the major known classes.

The following are elementary classes of \PP s.
\begin{theorem}\label{thm:elementary}
$\mbox{ }$
\begin{my_enumerate}
\item Every linear polynomial over \fq\ is a \PP\ of \fq.
\item The monomial $x^n$ is a \PP\ of \fq\ if and only if $\gcd(n,q-1)=1$.
\end{my_enumerate}
\end{theorem}

\begin{proof}
(1) Trivial. (2) Since $0^n=0$ the monomial $x^n$ is onto if and only if the function $f: \mfq^{\times} \to \mfq^{\times},x \mapsto x^n$ is onto. Let $g$ be a primitive element of the cyclic group $\mfq^{\times}$. Then the image of $\mfq^{\times}$ under $f$ is the cyclic subgroup generated by $g^n$, which equals $\mfq^{\times}$ if and only if $g^n$ is a primitive element. This is equivalent to the statement $\gcd(n,q-1)=1$.
\end{proof}

We now consider a class of polynomials known as $q$-polynomials. Let $q=p^r$ where $p$ is a prime and $r$ is a positive integer. Then a polynomial of the form
\[L(x)=\sum^n_{i=0} a_i x^{q^i}= a_0 x + a_1 x^q + \cdots + a_n x^{q^n} \in \mathbb{F}_{q^m}[x] \]
is called a {\bf \emph{q}-polynomial} over $\mathbb{F}_{q^m}$. Such polynomials are also known as \emph{linearised polynomials}, whose name stems from the properties
\begin{my_enumerate}
\item $L(\beta + \gamma) = L(\beta) + L(\gamma)$ for all $\beta, \gamma \in \mathbb{F}_{q^m}$,
\item $L(c \beta)=c L(\beta)$ for all $c \in \mathbb{F}_q, \beta \in \mathbb{F}_{q^m}$.
\end{my_enumerate}
We remark that properties (1) and (2) hold more generally for $\beta,\gamma$ in an arbitrary extension field of $\mathbb{F}_{q^m}$. If $\mathbb{F}_{q^m}$ is considered as a vector space over $\mathbb{F}_q$ then these properties show that $L(x)$ is a linear operator on $\mathbb{F}_{q^m}$.

The following theorem classifies when a \emph{p}-polynomial is a \PP.
\begin{theorem}\label{thm:p_poly}
Let \fqs be of characteristic $p$. Then the $p$-polynomial 
\[L(x)=\sum^m_{i=0} a_i x^{p^i} \in \mfqx \]
is a \PP \space if and only if $L(x)$ only has the root 0 in \fq.
\end{theorem}

\begin{proof}
Necessity is obvious. Suppose that $L(x)$ only has the root zero. Then by the discussion above we have $L(a)=L(b)$ if and only if $L(a-b)=0$. But since zero is the only root of $L(x)$ we must then have $a=b$. So $L(x)$ is one-to-one, so it is a \PP\ (\lref{lem:alt_defns}).
\end{proof}

We have a second criterion that applies to a class of $q$-polynomials.
\begin{theorem}
Let $\mathbb{F}_{q^m}$ be an extension of \fq\ and consider polynomials of the form
\[ L(x)=\sum^{m-1}_{i=0} a_i x^{q^i} \in \mathbb{F}_{q^m}[x]. \]
Then $L(x)$ is a \PP \space of $\mathbb{F}_{q^m}$ if and only if $\det(A) \neq 0$, where 
\[ A =
\begin{pmatrix}
a_0 & a_{m-1}^q & a_{m-2}^{q^2} & \cdots & a_1^{q^{m-1}} \\
a_1 & a_{0}^q & a_{m-1}^{q^2} & \cdots & a_2^{q^{m-1}} \\
a_2 & a_{1}^q & a_{0}^{q^2} & \cdots & a_3^{q^{m-1}} \\
\vdots & \vdots & \vdots & \ddots & \vdots \\
a_{m-1} & a_{m-2}^q & a_{m-3}^{q^2} & \cdots & a_0^{q^{m-1}} \\
\end{pmatrix}.
\]
If each $a_i$ is an element of \fq\ then $L(x)$ is a \PP \space of $\mathbb{F}_{q^m}$ if and only if 
\[ \gcd \left( \sum_{i=0}^{m-1} a_i x^i , x^i-1 \right) =1.
\]
\end{theorem}

If \fq\ is a finite field then polynomials that are \PP s of \emph{all} finite extensions of \fq\ are very rare. The following theorem gives the complete classification of polynomials with this property, which is in fact a special class of $p$-polynomial.
\begin{theorem}
Let $q=p^r$ where $p$ is a prime and $r$ is a positive integer. Then a polynomial $f \in \mfqx$ is a \PP\ of all finite extensions of \fq\ if and only if it is of the form $f(x)=ax^{p^h}+b$, where $a \neq 0$ and $h$ is a nonnegative integer.
\end{theorem}

\begin{proof}
Let $\mathbb{F}_{q^m}$ be a finite extension of \fq. If $c=a^{-1}b$ then we have
\[  f(x)=ax^{p^h}+b=a(x^{p^h}+c)=a(x+c)^{p^h}. \]
Then $f(x)=h \circ g$, where $g(x)=x+c$ is a \PP\ of $\mathbb{F}_{q^m}$ by \tref{thm:elementary} and $h(x)=a x^{p^h}$ is a \PP\ of $\mathbb{F}_{q^m}$ by \tref{thm:p_poly}. Hence, $f(x)$ is a \PP\ of $\mathbb{F}_{q^m}$. For necessity see \cite{lidl}.
\end{proof}

\begin{corollary}
If $f \in \mfqx$ is not of the form $f(x)=ax^{p^h}+b$ then there are infinitely many extension fields $\mathbb{F}_{q^m}$ of \fq\ such that $f$ is not a permutation polynomial of $\mathbb{F}_{q^m}$.
\end{corollary}

The next theorem gives a class of \PP s of a very specific form.
\begin{theorem}\label{thm:specific_class}
Let $h$ be a positive integer with $\gcd(h,q-1)=1$ and let $s$ be a positive divisor of $q-1$. Let $g \in \mfqx$ be such that $g(x^s)$ has no nonzero root in \fq. Then the polynomial
\[ f(x)=x^h (g(x^s))^{(q-1)/s} \]
is a \PP\ of \fq.
\end{theorem}

\begin{proof}
We use \tref{thm:hermite2}. Clearly condition (1) is satisfied. Let $1 \leq t \leq q-2$ and suppose that $s$ does not divide $t$. Now, all exponents of $f(x)^t$ are of the form $ht+ms$ for some positive integer $m$, and since $\gcd(h,s)=1$ none of these exponents is divisible by $s$. Hence no exponents are divisible by $q-1$. So there are no terms of the form $x^{i(q-1)}$ in the expansion of $f(x)^t$, so the reduction of $f(x)^t$ has degree $\leq q-2$. 

Now suppose that $t=ks$ for some positive integer $k$. Then we have 
\[ f(x)^t=x^{ht} (g(x^s))^{(q-1)k}. \]
For all $c \in \mfq^{\times}$ we have $f(c)=c^{ht}$ (because $g(c^s) \neq 0$), and for $c=0$ we have $f(0)=0=0^{ht}$. By \lref{lem:reduction} we have 
\[ f(x)^t \equiv x^{ht} \mod (x^q-x), \]
and since $q-1$ does not divide $ht$ the monomial $x^{ht}$ reduces modulo $x^q-x$ to a polynomial of degree $\leq q-2$.
\end{proof}

The following theorem completely classifies \PP s of the form $x^{(q+1)/2}+ax$ for odd $q$. As is the case with many families of permutation polynomials (see Table \ref{table:normalised}), whether or not a polynomial family parametrised by $a$ is a \PP\ often depends on the \emph{quadratic character} of $a$; that is, whether or not $a$ is a square in \fq. The following theorem is the first place we encounter this. 

We remind the reader that for all $x \in \mfq^{\times}$ we have
\begin{equation}\label{eq:squares}
x^{(q-1)/2} = 
\begin{cases}
1 &\text{if $x$ is a square}, \\
-1 &\text{if $x$ is a nonsquare}.
\end{cases}
\end{equation}

\begin{theorem}\label{thm:quadratic_binomial}
If $q$ is odd then the polynomial $x^{(q+1)/2}+ax \in \mfqx$ is a \PP\ of \fq\ if and only if $a^2-1$ is a nonzero square.
\end{theorem} 

\begin{proof}
Note that $f(x)=x^{(q+1)/2}+ax=(x^{(q-1)/2}+a)x$, so we have by \eref{eq:squares}
\begin{equation}\label{eq:f_squares}
 f(x)= 
\begin{cases}
(a+1)x &\text{if $x$ is a nonzero square}, \\
(a-1)x &\text{if $x$ is a nonsquare}, \\
0 &\text{if $x=0$}.
\end{cases}
\end{equation}
If $a^2-1=0$ then $a= \pm 1$, in which case $f(x)$ has repeated roots by \eref{eq:f_squares}. So we may assume that $a \not\in \{ 1,-1 \}$. Now \eref{eq:f_squares} shows that the image of \fq\ under $f$ is given by
\[ \{ (a-1)x \text{ : $x \in \mfq^{\times}$ is a nonsquare} \} \cup \{ (a+1)x \text{ : $x \in \mfq^{\times}$ is a square} \} \cup \{ 0 \}. \]
The first set contains precisely the squares in $\mfq^{\times}$ if $a-1$ is a square, and precisely the nonsquares in $\mfq^{\times}$ if $a-1$ is a nonsquare. Similarly, the second set contains precisely the nonsquares in $\mfq^{\times}$ if $a+1$ is a square, and precisely the squares in $\mfq^{\times}$ if $a+1$ is a nonsquare. Hence, $f(x)$ is onto \fq\ if and only if $a-1$ and $a+1$ are either both squares or both nonsquares. We can state this condition more compactly as
\[ (a-1)(a+1)=a^2-1 \text{ is a nonzero square.} \]
\end{proof}

The more general class of polynomials of the form $x^{(q+m-1)/m}+ax$, where $m$ is a positive divisor of $q-1$, have also been classified. 
\begin{theorem}
Let $m>1$ be a divisor of $q-1$. Then the polynomial $f(x)=x^{(q+m-1)/m}+ax \in \mfqx$ is a \PP\ of \fq\ if and only if $(-a)^m \neq 1$ and 
\[ \left( \frac{a+\xi^i}{a+\xi^j} \right)^{\frac{q-1}{m}} \neq \xi^{j-i} \text{ for all } 0 \leq i <j <m,\]
where $\xi$ is a fixed primitive $m^{th}$ root of unity in \fq.
\end{theorem}

We now introduce a class of polynomials known as \emph{Dickson polynomials}. 

\begin{newdef}
Let $R$ be a commutative ring with identity. For $a \in R$ define the {\bf Dickson polynomial} $g_k(x,a)$ of degree $k$ over $R$ by
\begin{equation*}
g_k(x,a) = \sum_{j=0}^{\lfloor k/2 \rfloor} \dfrac{k}{k-j} \begin{pmatrix} k-j \\ j \end{pmatrix} (-a)^j x^{k-2j}.
\end{equation*}
\end{newdef}

Dickson polynomials satisfy a number of interesting properties, for example we have $g_1(x,a)=x$, $g_2(x,a)=x^2-2a$, and
\[ g_{k+1}(x,a)=x g_k (x,a)-a g_{k-1}(x,a), \text{ for } k \geq 2.\]
We refer the reader to \cite{lidl} for more interesting properties of $g_k(x,a)$. The following theorem characterises when Dickson polynomials are \PP s. Remarkably, whether or not the Dickson polynomial $g_k(x,a)$ is a \PP\ of \fq\ depends only on its degree (not on $a$).

\begin{theorem}
Let  $a \in \mfq^{\times}$. Then the Dickson polynomial $g_k(x,a)$ is a \PP \space of \fqs if and only if $\gcd(k,q^2-1)=1$.
\end{theorem}
An interesting perspective of Dickson polynomials is that they generalise the power polynomial $x^k$. Because $g_k(x,0) = x^k$, which by \tref{thm:elementary} is a \PP\ of \fq\ if and only if $\gcd (k,q-1)=1$. On the other hand, if $a \neq 0$ then the polynomial $g_k(x,a)$ is a \PP\ of \fq\ if and only if $\gcd (k,q^2-1)=1$.

In this section we have endeavoured to list the major known classes of \PP s, but note that we have not attempted an exhaustive survey. Other classes of polynomials have also been characterised; for example, see \cite{wan1991} for polynomials of the form $x^hf(x^{\frac{q-1}{d}})$, and \cite{wan1994} for binomials of the form $x^{m+ \frac{q-1}{2}}+ax^m$.

\section{Normalised Permutation Polynomials}
Let $q=p^r$ and let $f(x)=\sum_{i=0}^n a_i x^i$ be a \PP\ of \fq. Note that the set of \PP s of \fq\ is closed under composition, so in particular the polynomial $g(x) = c f(x+b)+d$ is a \PP\ of \fq\ for all choices of $b,c,d \in \mfq$, $c \neq 0$. Expanding $g$, we have
\[
g(x) =c a_n x^n + c(a_n b n + a_{n-1} )x^{n-1} + \cdots + c(a_n b^n + a_{n-1} b^{n-1} + \cdots + a_0) +d.
\]
By suitable choices of $c$ and $d$ we can ensure that $g$ is monic and satisfies $g(0)=0$; that is, if we choose
\begin{equation}\label{eq:normalised_cd}
 c= a_n^{-1} \text{ and } d=-c (a_nb^n+a_{n-1}b^{n-1}+\cdots+a_0).
\end{equation}
In addition, if $n \not\equiv 0 \mod p$ then we can remove the $x^{n-1}$ term by setting 
\begin{equation}\label{eq:normalised_b} 
b=-a_{n-1}/(a_n n).
\end{equation}
This motivates the following definition.
\begin{newdef}\label{def:normalised}
A \PP\ $f\in \mfqx$ is said to be of {\bf normalised form} if $f$ is monic, $f(0)=0$, and when the degree $n$ of $f$ is not divisible by the characteristic of \fq, the coefficient of $x^{n-1}$ is zero.
\end{newdef}

\begin{remark}
If $p \nmid n$ then any \PP\ $f\in \mfqx$ of degree $n$ has a unique normalised representative $g \in \mfqx$ given by 
\[g(x)=c f(x+b)+d, \]
with $b,c,d$ as defined in \eref{eq:normalised_cd} and \eref{eq:normalised_b}. If $p \mid n$ then with the convention $b=0$ every \PP\ $f\in \mfqx$ of degree $n$ has a unique normalised representative $h \in \mfqx$ given by 
\[h(x)=c f(x)+d, \]
with $c,d$ as defined in \eref{eq:normalised_cd}.\qed
\end{remark}

If we divide all \PP s of \fq\ into classes based on their unique normalised representatives then we have a partition of the set of \PP s of \fq. By counting the number of polynomials in each partition we give an enumerative proof of a classical result in number theory known as Wilson's theorem.

If $q$ is a prime power then there are exactly $q!$ \PP s of \fq\ of degree $<q$. We wish to count the number of \PP s represented by each normalised \PP\ $g \in \mfqx$. First consider the monomial $g(x)=x$. This is the only normalised \PP\ of degree 1 and is the representative of all linear \PP s of the form 
\[ cx+d, \text{ where } c,d \in \mfq, c\neq 0.\]
Hence, $g(x)=x$ represents $q(q-1)$ \PP s. If $g$ is a normalised \PP\ with $\deg(g)>1$ not divisible by $p$, then $g$ represents the $q^2(q-1)$ \PP s given by
\[ c g(x+b)+d, \text{ where } b,c,d \in \mfq, c\neq 0.\]
If, on the other hand, $\deg(g)>1$ and $p$ divides $\deg(g)$, then $g$ represents the $q(q-1)$ \PP s given by 
\[c g(x)+d, \text { where } c,d \in \mfq, c\neq 0.\]
Hence, if $k_1$ is the number of nonlinear normalised \PP s with degree prime to $p$, and $k_2$ is the number of normalised \PP s with degree divisible by $p$, then we have the identity
\begin{equation}\label{eq:for_wilsons}
q! = q(q-1) (1+k_2 + q k_1).
\end{equation}
Using this identity we give the enumerative proof from \cite{dickson1897II} of the following theorem.
\begin{theorem}{\bf(Wilson's theorem).}
If $n$ is a positive integer then the identity
\[(n-1)! \equiv -1 \mod n \]
holds if and only if $n$ is prime.
\end{theorem}

\begin{proof}
Let $p$ be a prime and consider the set of \PP s of $\mathbb{F}_p$. By \lref{lem:reduction} and \cref{cor:max_degree} we may assume that the degrees of all polynomials are less than $p-1$, thus not being multiples of $p$. By \eref{eq:for_wilsons} we then have $p!=p(p-1)(1+pk)$ for some positive integer $k$. Dividing by $p$ and reducing mod $p$ we have
\[ (p-1)! \equiv -1 \mod p.\]
On the other hand, if $n$ is composite then $(n-1)! \equiv 0 \mod p$ for all prime factors $p$ of $n$, so $(n-1)! \not\equiv -1 \mod n$. 
\end{proof}

\chapter{The Carlitz Conjecture}\label{chap:carlitz}
In an invited address before the Mathematical Association of America in 1966, Professor L. Carlitz presented a conjecture that would motivate almost 30 years of research and significant interest in permutation polynomials. It had been known since 1897 that there exist \PP s of degree 1, 3 and 5 over infinitely many fields \fq, but excepting fields of even characteristic there exist only finitely many \PP s of degree 2, 4 or 6  (see Table \ref{table:normalised}, \chref{chap:deg6} and \cite{dickson1897II}). Carlitz conjectured that perhaps this behaviour was typical; that is, except for fields of small order there are no \PP s of even degree over fields of odd characteristic.

Although there was immediate success in some special cases, progress was made slowly over the next three decades until Carlitz's conjecture was finally resolved in the affirmative by Fried, Guralnick and Saxl in 1993. The story does not end there, however, for around the same time as the work of Fried \emph{et al.} two separate generalisations of Carlitz's conjecture were published. The first, due to Wan, was shortly confirmed. However, a second generalisation conjectured by Mullen has been discussed in published literature but until now has remained unresolved. In Section \ref{sec:mullen} we provide a counterexample to Mullen's conjecture, and also point out how recent results imply an altered version of its statement.

The main goal of this chapter is to give a survey of the major results leading to the proofs of the Carlitz conjecture and Wan's generalisation. We also aim to give some of the history of this journey, and disprove the aforementioned conjecture of Mullen. We will see that the proof of Carlitz's conjecture is closely linked with with the notion of \emph{exceptional polynomials}. These polynomials are discussed in Section \ref{sec:exceptional} along with their relationship with permutation polynomials. We will also be concerned with the so-called \emph{value set} of a polynomial, defined as follows: if $f \in \mfqx$ is a polynomial then the \emph{value set of f}, denoted $V_f$, is given by
\[ V_f = \{ f(c) : c \in \mfqx \}. \]
Note that $f$ is a permutation polynomial of \fq\ if and only if $| V_f |=q$.

\section{Exceptional Polynomials}\label{sec:exceptional}
Let $F$ be a field and recall that we have unique factorisation in $F[x_1,x_2,...,x_n]$ into irreducibles.
\begin{newdef}
A polynomial $f \in F[x_1,x_2,...,x_n]$ is called {\bf absolutely irreducible} if it is irreducible over every algebraic extension of $F$.
\end{newdef}
Equivalently, $f \in F[x_1,x_2,...,x_n]$  is absolutely irreducible if it is irreducible over the algebraic closure of $F$.

\begin{example}
The polynomial $f(x)=x^2+1 \in \mathbb{F}_7[x]$ is irreducible, but not absolutely irreducible because it factors as $(x-\sqrt{-1})(x+\sqrt{-1})$ over $\mathbb{F}_7(\sqrt{-1})=\mathbb{F}_{7^2}$. In fact, it is easy to see that a univariate polynomial $f \in \mfqx$ is absolutely irreducible if and only if it is a linear polynomial. \qed
\end{example}

\begin{example}
The polynomial $g(x,y)=x^2+y^2 \in \mathbb{F}_7[x,y]$ is irreducible, but factors as $g(x,y)=(x+ \sqrt{-1}y)(x- \sqrt{-1}y)$ over $\mathbb{F}_7(\sqrt{-1})=\mathbb{F}_{7^2}$. However, the polynomial $h(x,y)=x^2-y^3 \in \mathbb{F}_7[x,y]$ is absolutely irreducible. \qed
\end{example}

We now introduce \emph{exceptional polynomials}, which are closely related with permutation polynomials.
\begin{newdef}
A polynomial $f \in \mfqx$ of degree $\geq 2$ is said to be {\bf exceptional} over \fq \space if no irreducible factor of
\[ \Phi(x,y) = \dfrac{f(x)-f(y)}{x-y} \]
in $\mfq [x,y]$ is absolutely irreducible.
\end{newdef}
Equivalently, $f$ is exceptional if every irreducible factor of $\Phi(x,y)$ becomes reducible over some algebraic extension of \fq.

Exceptional polynomials were first introduced by Davenport and Lewis in \cite{davenport1963}, where the authors also conjectured the following relationship between exceptional polynomials and permutation polynomials. Although special cases were proved by MaCleur \cite{maccleur1967} and by Williams \cite{williams1968}, first general proof was given by Cohen \cite{cohen1970} using deep methods of algebraic number theory.

\begin{theorem}\label{thm:all_exceptional_pps}
Every exceptional polynomial over \fqs is a permutation polynomial of \fq.
\end{theorem}

In \cite{wan1991a}, D. Wan shows that \tref{thm:all_exceptional_pps} is a consequence of the following result, which states that any polynomial producing sufficiently many distinct elements is a \PP. Wan proves this theorem by way of a $p$-adic lifting lemma, but we present here the more elementary proof from \cite{turnwald1995} based on elementary symmetric polynomials.

Recall the following about symmetric polynomials. If $R$ is a ring then a polynomial $f \in R[x_1,...,x_n]$ is called \emph{symmetric} if $f(x_{i_1},...,x_{i_n})=f(x_1,...,x_n)$ for any permutation $i_1,...,i_n$ of the integers $1,...,n$. If $z$ is an indeterminate over $R[x_1,...,x_n]$ then the \emph{$k^{th}$ elementary symmetric polynomial} $s_k$ is defined by
\[ \prod_{i=1}^n (z-x_i) = \sum_{k=0}^n (-1)^k s_k z^{n-k}. \]
That is, $s_0=1$ and
\[
s_k(x_1,...,x_n)=\sum_{1 \leq i_1 < \cdots < i_k \leq n} x_{i_1}\cdots x_{i_k} \text{ for all } 1 \leq k \leq n.
\]
The \emph{fundamental theorem on symmetric polynomials} states that every symmetric polynomial $f(x_1,...,x_n)$ is a polynomial in $s_1(x_1,..,x_n),...,s_n(x_1,..,x_n)$.

If $R=\mfq$ is a finite field then it is easy to see that $\prod_{i=1}^q (x-c_i) = x^q-x$ if and only if $\{c_1,...,c_q \} = \mfq$. Hence, if $\{c_1,...,c_q \} = \mfq$ then we have $s_k(c_1,c_2,...,c_q)=0$ for all $1 \leq k \leq q-2$. Also note the identity 
\[s_k(x,x,...,x)=\sum_{1 \leq i_1 < \cdots < i_k \leq q} x^k = \left( \begin{matrix} q \\ k \end{matrix} \right) x^k = 0 \text{ for all } 1 \leq k \leq q-1. \]

\begin{theorem}{\emph{\bf (Wan).}}\label{thm:wan_bound}
Let $f(x) \in \mfqx$ be a polynomial of positive degree $n$. If $f(x)$ is not a \PP\ of \fq, then
\[ | V_f | \leq q- \left\lceil \frac{q-1}{n} \right\rceil, \]
where $\lceil m \rceil$ denotes the least integer $\geq m$.
\end{theorem}

\begin{proof}
If $n \geq q$ then $\left\lceil \frac{q-1}{n} \right\rceil=1$ and the assertion holds trivially, so assume that $1 \leq n \leq q-1$. Then $| V_f | \geq 2$, for otherwise $f$ is constant on all values of \fq\ in contradiction to \lref{lem:unique_poly}.

Let $\mfq=\{ c_1,...,c_q \}$ and let $s_k$ represent the $k^{th}$ elementary symmetric polynomial of the values of $f(x)$, that is, 
\[ \prod_{i=1}^q (x-f(c_i)) = \sum^q_{k=0} (-1)^k s_k x^{q-k}. \]
Let $u$ be the least positive integer $k$ such that $s_k \neq 0$ if such $k$ exists; otherwise let $u= \infty$. 

Suppose that $k$ is such that $0<kn<q-1$ and consider the symmetric polynomial $s_k(f(x_1), f(x_2),...,f(x_q))$. This polynomial has degree at most $kn<q-1$, so by the fundamental theorem on symmetric polynomials it is a polynomial in $s_1(x_1,...,x_q),...,s_{q-2}(x_1,...,x_q)$. Hence, $s_k(f(c_1),...,f(c_q))$ is a polynomial in $s_1(c_1,...,c_q),...,s_{q-2}(c_1,...,c_q)$, all of which are zero. The constant term is $s_k(f(0),...,f(0))=0$. Hence, 
\begin{equation}\label{eq:wan_proof1}
u \geq (q-1)/n.
\end{equation}
Consider the polynomial  
\[ g(x)= x^q - x -\prod_{i=1}^q (x-f(c_i)). \]
Since $\deg \left( x^q -\prod_{i=1}^q (x-f(c_i)) \right) = q-u$ we have $\deg(g) \leq q-u$. Now $g(x)=0$ if and only if $\prod_{i=1}^q (x-c_i) = x^q-x$, which is equivalent to $f$ being a \PP. Hence, if $f$ is not a \PP\ then $g(x) \neq 0$. But then $f(c_i)$ is a root of $g$ for all $1 \leq i \leq q$, so 
\begin{equation}\label{eq:wan_proof2}
| V_f | \leq \deg(g) \leq q-u.
\end{equation}

Combining \eref{eq:wan_proof1} and \eref{eq:wan_proof2} we have
\[ | V_f | \leq q- \left\lceil \frac{q-1}{n} \right\rceil. \]
\end{proof}

Note that \tref{thm:wan_bound} is the precisely the statement given in \tref{thm:9criteria} (5). For Wan's proof that \tref{thm:wan_bound} implies \tref{thm:all_exceptional_pps}, see \cite[Theorem~5.1]{wan1991a}.

It is true that all exceptional polynomials are \PP s, so the converse question naturally arises: are all \PP s exceptional? The following example shows that this is not the case.
\begin{example}\label{eg:nonexceptionalpp}
Let $q=p^r$, $a \in \mfq$, and consider the polynomial
\begin{equation}\label{eq:pp_notexceptional}
f(x)=x^p +a \in \mfq[x].
\end{equation}
Then $f$ is a \PP\ by \tref{thm:p_poly}, but we have
\[ \Phi(x,y)=\frac{x^p-y^p}{x-y}=\frac{(x-y)^p}{x-y}=(x-y)^{p-1}. \]
All irreducible factors $\Phi$ are linear, thus being irreducible over every algebraic extension of \fq. Hence, $f$ is not exceptional. \qed
\end{example}

The polynomial in \eref{eq:pp_notexceptional} is a permutation polynomial of $\mathbb{F}_{p^r}$ for all positive integers $r$. Hence, there exist examples of non-exceptional \PP s over fields of arbitrarily large order. However, such examples only arise for polynomials that are \emph{not separable}. We will see that excluding these troublesome polynomials it is true that all \PP s are exceptional - provided that $q$ is sufficiently large compared to the degree of the polynomial.

Note that for any $f \in \mfqx$ there exists a unique integer $t \geq 0$ and a polynomial $g \in \mfqx$ such that $f(x)=g(x^{p^t})$, but $f(x) \neq h(x^{p^{t+1}})$ for any $h \in \mfqx$. Then $t>0$ if and only if $f'(x)=0$. This motivates the following definition.
\begin{newdef}
A polynomial $f \in \mfqx$ is called {\bf separable} if $f'(x) \neq 0$.
\end{newdef}
We remark that in other areas of mathematics there exist different definitions of separable polynomials, but in the study of permutation polynomials the above definition is standard (see for example \cite{wan1990,gathen1991,cohen1991,cohen1995}). 

Note that if $f(x)$ is not separable then we can write $f(x)=g(x^{p^t})$, where $t>0$ and $g \in \mfqx$ is separable. Then $f$ is a \PP\ if and only if $g$ a \PP, so in most cases we can assume without loss that polynomials are separable (otherwise we could replace $f$ with $g$).

If we assume separability then it is true that, apart from fields of small order, all \PP s are exceptional polynomials. The following result was proved by Wan \cite{wan1987} using a powerful theorem of Lang and Weil on the number of rational points of an algebraic curve over a finite field.  
\begin{theorem}\label{thm:sequence_exceptional}
There exists a sequence $c_1,c_2,...$ of integers such that for any separable polynomial $f \in \mfqx$ of degree $n$ we have: if $q \geq c_n$ and $f$ is a \PP\ then $f$ is exceptional over \fq.
\end{theorem}
We note that \tref{thm:sequence_exceptional} had already been proved by Hayes \cite{hayes1967} for polynomials satisfying $\gcd(n,q)=1$. The special case of Hayes' theorem when $q$ is prime was established by Davenport and Lewis \cite{davenport1963}; quantitative versions were given by Bombieri and Davenport \cite{davenport1966} and Tiet{\"{a}}v{\"{a}}inen \cite{tiet1966}.

Although versions of \tref{thm:sequence_exceptional} had been known for over 20 years, in 1991 von zur Gathen \cite{gathen1991} proved the following version with the explicit sequence $c_n=n^4$. In the language of the previous discussion this quantifies what is meant by a field of `small order'. As in the work of Hayes and Wan a central ingredient of von zur Gathen's proof is the Lang and Weil theorem. This result ultimately allows the Carlitz conjecture to be stated quantitatively.
\begin{theorem}\label{thm:ppsexceptional} 
Let $f \in \mfqx$ be separable of degree $n$. If $q \geq n^4$ and $f$ is a \PP, then $f$ is exceptional.
\end{theorem}
In light of \tref{thm:ppsexceptional} we have the following results on the \emph{non existence} of \PP s of \fq\ of certain degrees $n$. See \cite[Ch.~7]{lidl} for proofs of their non-quantitative analogues; the bound $q \geq n^4$ comes from \tref{thm:ppsexceptional}, see \cite[Corollary~3]{gathen1991}.

\begin{theorem}\label{thm:nopps_rootof1}
Let $n \geq 1$ and suppose that $q \geq n^4$. If $\gcd(n,q)=1$ and \fq\ contains an $n^{th}$ root of unity different from 1 then there is no \PP\ of \fq\ of degree $n$.
\end{theorem}

\begin{corollary}
Suppose that $n$ is positive and even, $q \geq n^4$ and $\gcd(n,q)=1$. Then there is no \PP\ of \fq\ of degree $n$.
\end{corollary}

\begin{proof}
Let $\zeta=-1$ in \tref{thm:nopps_rootof1}.
\end{proof}

\begin{corollary}
Suppose that $q \geq n^4$ and $\gcd(n,q)=1$. Then there exists a \PP\ of \fq\ of degree $n$ if and only if $\gcd(n,q-1)=1$.
\end{corollary}

\begin{proof}
If $\gcd(n,q-1)=1$ then the monomial $x^n$ is a \PP\ of \fq\ by \tref{thm:elementary}. Conversely, suppose that $\gcd(n,q-1)=d>1$. Since the multiplicative group $\mfq^{\times}$ is cyclic of order $q-1$ it follows that $g^{(q-1)/d}$ is an $n^{th}$ root of unity different to 1, where $g$ is a primitive element of \fq.
\end{proof}

\section{A Conjecture of Carlitz}
In an invited address before the Mathematics Association of America in 1966, Professor L. Carlitz made the following conjecture:

\begin{prop}{\emph{\bf (Carlitz conjecture).}}\label{prop:carlitz}
For every even positive integer $n$, there is a constant $c_n$ such that for each finite field of odd order $q > c_n$, there does not exist a \PP\ of \fq\ of degree $n$.
\end{prop}
This proposition was the motivation for many papers and generated much interest over the following three decades. A chronology of the major results leading to the proof of \pref{prop:carlitz} is given below. Carlitz' conjecture was finally resolved in the affirmative by Fried, Guralnick and Saxl in 1993. They used the classification of finite simple groups to prove, in particular, the following theorem.
\begin{theorem}{\bf (Fried \textit{et al.}).}\label{thm:fried}
If $q$ is odd then every exceptional polynomial over \fq\ has odd degree.
\end{theorem}
In light of \tref{thm:ppsexceptional} this confirms the Carlitz conjecture. We state this as a theorem.
\begin{theorem}\label{thm:carlitz}
Let $n$ be a positive even integer and suppose that $q \geq n^4$ is odd. Then there does not exist a \PP\ of \fq\ of degree $n$.
\end{theorem}

\begin{proof}
Let $n$ be an even positive integer and let $q=p^r \geq n^4$ be odd. Suppose that $f \in \mfqx$ has degree $n$. We may assume that $f$ is separable. For otherwise, write $f(x)=g(x^{p^t})$, where $g'(x) \neq 0$ and $t$ is a positive integer. Then $g \in \mfqx$ is separable of even degree $m=n/p^t$, and $q \geq m^4 = n^4/p^{4t}$. Moreover, $f$ is a \PP\ of \fq\ if and only if $g$ is, so we may replace $f$ by $g$ and $n$ by $m$. Hence, assume that $f$ is separable. If $f$ is a \PP\ then it is exceptional by \tref{thm:ppsexceptional}, but then $n$ is odd by \tref{thm:fried}, a contradiction.
\end{proof}

The following is a timeline of the major results leading to the proof of \tref{thm:carlitz}.
\begin{description}[leftmargin=15mm,style=sameline,itemsep=1pt,font=\normalfont] 
\item[\emph{1897}] Dickson's list (Chapter \ref{chap:deg6}, Table \ref{table:normalised},\cite{dickson1897II}) shows that there are only finitely many fields \fq\ of odd characteristic containing \PP s of degree $n=2,4,6$.
\item[\emph{1966}]  Carlitz presents his conjecture during an address to the MAA.
\item[\emph{1967}] Hayes \cite{hayes1967} proves the conjecture for $n=8,10$ and the general case case $p \nmid n$.
\item[\emph{1973}] Lausch and Nobauer \cite[p.~202]{lausch1973} prove the conjecture for $n=2^m$.
\item[\emph{1987}] Wan \cite{wan1987} proves the conjecture for $n=12,14$ and states an equivalent version in terms of exceptional polynomials.
\item[\emph{1988}] Lidl and Mullen \cite[P9]{lidl1988} feature the Carlitz conjecture as an unsolved problem.
\item[\emph{1990}] Wan \cite{wan1990} proves the conjecture for $n=2r$, where $r$ is an odd prime.
\item[\emph{1991}] Independently to Wan, and almost at the same time, Cohen \cite{cohen1991} proves the case $n=2r$, $r$ an odd prime. In addition, he proves the conjecture for all $n < 1000$. 
\item[\emph{1991}] von zur Gathen \cite{gathen1991} proves \tref{thm:ppsexceptional}, allowing the Carlitz conjecture to be stated quantitatively.
\item[\emph{1993}] Carlitz's conjecture is proven in general by Fried, Guralnick and Saxl \cite{fried1993}.
\end{description}

\begin{figure}[h!]
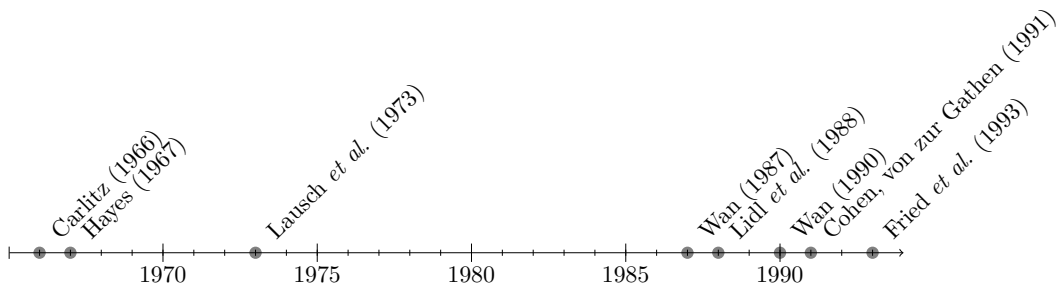

\begin{chronology}[5]{1965}{1993}{3ex}{\textwidth}
\event{1966}{Carlitz (1966)}
\event{1967}{Hayes (1967)}
\event{1973}{Lausch \emph{et al.} (1973)}
\event{1987}{Wan (1987)}
\event{1988}{Lidl \emph{et al.} (1988)}
\event{1990}{Wan (1990)}
\event{1991}{Cohen, von zur Gathen (1991)}
\event{1993}{Fried \emph{et al.} (1993)}
\end{chronology}
\caption{\emph{A timeline of major results leading to the proof of the Carlitz conjecture.}}
\end{figure}

\section{Wan's Generalisation}\label{sec:wan}
Coinciding with the time that Fried \emph{et al.} proved Carlitz's original conjecture, in 1993 Wan proposed the following generalisation \cite{wan1993}.
\begin{prop}{\emph{\bf (Carlitz-Wan conjecture).}}\label{prop:carlitz_wan}
Let $q > n^4$. If $\gcd(n,q-1)>1$, then there are no \PP s of degree $n$ over \fq.
\end{prop}
Recall that if $\gcd(n,q-1)=1$ then there exist \PP s of degree $n$ (for example the monomial $x^n$). \pref{prop:carlitz_wan} can be interpreted as a partial converse of this statement; that is, if $q > n^4$ then there exist \PP s of degree $n$ if and only if $\gcd(n,q-1)=1$. In the special case that $n$ is even and $q$ is odd, \pref{prop:carlitz_wan} reduces to the Carlitz conjecture (\pref{prop:carlitz}).

The work by Fried \emph{et al.} in \cite{fried1993}, which proved Carlitz's original conjecture, also proved the Carlitz-Wan conjecture for fields of characteristic $p>3$. The remaining special cases did not remain unresolved for long, for the following theorem by Lenstra implies \pref{prop:carlitz_wan} in full generality. See \cite{cohen1995} for a discussion of Lenstra's proof and an elementary version.
\begin{theorem}{\bf (Lenstra).}\label{thm:lenstra} 
Suppose $\gcd(n,q-1)>1$. Then there is no exceptional polynomial of degree $n$ over \fq.
\end{theorem}
We state Wan's generalisation of the Carlitz conjecture as a theorem.
\begin{theorem}\label{thm:carlitz_wan}
Let $q > n^4$. If $\gcd(n,q-1)>1$, then there are no \PP s of degree $n$ over \fq.
\end{theorem}

\begin{proof}
Note that as in the proof of \tref{thm:carlitz} we may assume without loss that all polynomials are separable. The result follows from \tref{thm:lenstra} and \tref{thm:ppsexceptional}.
\end{proof}

\section{On a Conjecture of Mullen}\label{sec:mullen}
The following generalisation of Carlitz's conjecture by Mullen appeared in \cite{mullen1991} and is discussed in \cite{wan1991a,turnwald1995,mullen1995,wan1993}. Until now it is an unresolved problem in published literature.

Based on computer calculations, Mullen proposed that if $n$ is even and $q$ is odd and sufficiently large then no polynomial is ``close'' to being a \PP . 
\begin{conj}{\bf (Mullen).}\label{conj:mullen}
If $n$ is even, $q$ is odd with $q>n(n-2)$ and $f \in \mfqx$ has degree $n$, then
\[ | V_f | \leq q - \left\lceil \frac{q-1}{n} \right\rceil . \]
\end{conj}
In light of \tref{thm:wan_bound} and \tref{thm:carlitz} (both appearing after Mullen's conjecture) we know this to be true for all $q \geq n^4$. We present a counterexample to Mullen's conjecture as stated. Let $a$ be an arbitrary nonzero element of $\mathbb{F}_{3^3}$ and consider the polynomial
\[ f(x)= x^6+a x^5-a^4 x^2 \in \mathbb{F}_{3^3}[x]. \]
Then $f$ is a \PP\ of $\mathbb{F}_{3^3}$, as proved in Section \ref{sec:ppsF9} and \cite{dickson1897II}. This contradicts Conjecture \ref{conj:mullen}, because $27=q>n(n-2)=24$, but 
\[ \mid V_f \mid = 27 \not\leq 22=q -\left\lceil \frac{q-1}{n} \right\rceil .\]
Armed with results published after Mullen's conjecture (\tref{thm:wan_bound} and \tref{thm:carlitz_wan}) we can give the following generalisation of Conjecture \ref{conj:mullen}, although the bound on $q$ is considerably weakened. In the special case that $n$ is even and $q$ is odd, this reduces (albeit with a looser bound) to Mullen's conjecture.
\begin{theorem}\label{thm:generalised_mullen}
If $\gcd(n,q-1)>1$ with $q>n^4$ and $f \in \mfqx$ has degree $n$, then
\[ | V_f | \leq q - \left\lceil \frac{q-1}{n} \right\rceil . \]
\end{theorem}

\begin{proof}
\tref{thm:carlitz_wan} and \tref{thm:wan_bound}.
\end{proof}

It would be interesting if future research could further reduce the bound in \tref{thm:generalised_mullen}; by the above discussion the true bound lies between $n(n-2)$ and $n^4$.

\chapter{Permutation Polynomials of Degree 6}\label{chap:deg6}
In 1897 Leonard Eugene Dickson \cite{dickson1897II} claimed to give, aside from degree 6 polynomials in even characteristic, a \emph{complete list of all reduced quantics of degree $\leq 6$ which are suitable to represent substitutions}. In modern parlance this is a claim to a complete list of normalised permutation polynomials (compare \cite[\textsection16]{dickson1897II} to \dref{def:normalised}). Historically, Dickson's claim has been largely accepted in literature \cite{lidl,hayes1967,wan1987,wan1990,li2010}, however in more recent times some doubts have been cast on this assertion. Though his classification of polynomials of degree less than 6 is still trusted, some authors have questioned the completeness of his characterisation of the degree 6, odd characteristic case. Indeed, \cite{evans1992} refers to this as a `partial list'. 

The main problem with verifying Dickson's claim is that his published proof in \cite{dickson1897II} is very difficult to follow. To his credit, the author did a remarkable job in deriving and solving the necessary sets of long, unfriendly equations without so much as a pocket calculator. However, his long and tricky proof is not easily accessible to the modern mathematician for a number of reasons, the main factor being that his language, notation and terminology are somewhat antiquated 115 years later. Furthermore, as is natural for a paper written before modern computing and printing, there are some unhelpful typographical errors and inconsistent notations. For these reasons, it has not been easy for the modern mathematician to verify Dickson's claim to a complete classification.

In this chapter we recreate in full detail the classical result of Dickson by deriving the full characterisation of degree 6 permutation polynomials in odd characteristic. The aim of this chapter is to finally put to rest the classification problem for permutation polynomials of degree $\leq 6$. Though our general ideas and methods are essentially the same as in \cite{dickson1897II}, we have not attempted to recreate Dickson's proof step-by-step. Indeed, in many ways our proofs are different to those presented in \cite{dickson1897II}. We deliberately give most details, for we feel that many of the rearrangements and tricks used in solving sets of equations are nonobvious. Our goal is a proof that can be easily followed in full detail by those who are unconvinced by Dickson's claim of a complete characterisation. Hopefully the arguments presented are more easily accessible to the modern mathematician than those in the original paper. 

Somewhat surprisingly, we find that the list given in \cite{dickson1897II} is, albeit with minor errors, indeed a full classification. We are, however, able to improve Dickson's list in several ways. In \cite{dickson1897II}, we note that not all \emph{normalised} permutation polynomials of degree 6 in characteristic 3 are listed. Instead, some of Dickson's polynomials have been reduced further than specified in \dref{def:normalised}. We are able to rectify this, and we suggest that confusion over this point is perhaps the reason that Dickson's list has been recently questioned. Furthermore, we clear up some errors in the list, and give a much cleaner parametrisation of one of the entries. In light of a very recent paper by Li \emph{et al.} \cite{li2010}, which lists all degree 6 and 7 \PP s over fields of characteristic 2, this completes the classification problem of \PP s of degree $\leq 6$.

\section{Some General Results}
\subsection{The Multinomial Theorem Modulo $p$}
We begin by defining multinomial coefficients, which the next theorem shows are analogous to the well-known binomial coefficients.

\begin{newdef}
If $t,n,k_1,...,k_n$ are nonnegative integers with $k_1+\cdots+k_n=t$ and $n \geq 2$, then define the {\bf multinomial coefficient} $\left( \begin{smallmatrix} t \\ k_1,k_2,...,k_n \end{smallmatrix} \right)$ to be
\[ 
\begin{pmatrix}
t \\
k_1,k_2,...,k_n
\end{pmatrix}
=
\frac{t!}{k_1!k_2!\cdots k_n!}.
\]
\end{newdef}

The following theorem is known as the multinomial theorem.
\begin{theorem}\label{thm:multinomial}
We have the following expansion:
\begin{equation*}
(x_1 + \cdots + x_n)^t = \sum_{\substack{k_1+ \cdots + k_n=t \\ k_1 \geq 0, \cdots, k_n \geq 0}} 
\begin{pmatrix}
t \\
k_1,...,k_n
\end{pmatrix} 
x_1^{k_1} \cdots x_n^{k_n}.
\end{equation*}
\end{theorem}
The following is the multinomial analogue of a classical theorem of Lucas. Its proof can be found in \cite[\textsection14-15]{dickson1897II}.
\begin{theorem}\label{thm:lucas}
Let $p$ be a prime and $k_1,k_2,...,k_n,t$ be nonnegative integers such that $k_1+k_2+\cdots+k_n=t$. Suppose that we have the following $p$-adic expansions:
\begin{align*}
k_i &= b_{i0} + b_{i1}p + b_{i2}p^2 + \cdots +b_{is}p^s \text{ for all } 1 \leq i \leq n, \\
t &= c_0 + c_1p + c_2 p^2 + \cdots + c_s p^s,
\end{align*}
where $0 \leq c_{j},b_{ij} \leq p-1$ for all $0 \leq j \leq s$ and $1 \leq i \leq n$. Then
\[
\begin{pmatrix}
t \\
k_1,k_2,...,k_n
\end{pmatrix}
\not\equiv 0 \mod p \text{ if and only if } 
\sum_{i=1}^n b_{ij} = c_j \text{ for all } 0 \leq j \leq s.
\]
If $\left( \begin{smallmatrix} t \\ k_1,k_2,...,k_n \end{smallmatrix} \right) \not\equiv 0 \mod p$ then we have
\[ 
\begin{pmatrix}
t \\
k_1,k_2,...,k_n
\end{pmatrix}
\equiv
\begin{pmatrix}
c_0 \\
b_{10},b_{20},...,b_{n0}
\end{pmatrix}
\cdots
\begin{pmatrix}
c_s \\
b_{1s},b_{2s},...,b_{ns}
\end{pmatrix}
\mod p.
\]
\end{theorem}

\subsection{A General Restriction on Coefficients}
The following theorem shows that any normalised \PP\ of degree $n$ of \fq, where $q \equiv -1 \mod n$, has no $x^{n-2}$ term.
\begin{theorem}\label{thm:thirdhighestcoeffzero}
Let 
\[ f(x) = x^n+a_{n-2} x^{n-2}+\cdots+ a_1 x \in \mfqx \]
be a normalised \PP\ of \fq, where $3\leq n \leq q-2$ and $q=p^r\equiv-1 \mod n$. Then $a_{n-2}=0$.
\end{theorem}

\begin{proof}
Note that $p^r \equiv -1 \mod n$ implies that $p \nmid n$, so $f$ is indeed the general form for a normalised \PP\ of degree $n$ (\dref{def:normalised}).

Write $q=nm-1$; then it is clear that $m=(q+1)/n \not\equiv 0 \mod p$. We also have
\[ 1< m=\frac{q+1}{n} \leq q-2, \]
because $1 < (q+1)/n$ is equivalent to $q > n-1$ and $(q+1)/n \leq q-2$ is equivalent to $q \geq (2n+1)/(n-1)$, and both of these conditions hold under the assumption $3\leq n \leq q-2$. Hence, by \tref{thm:hermite}, the reduction of $f(x)^m$ modulo $x^q-x$ has degree $\leq q-2$.

To find the coefficient of $x^{q-1}$ in $f(x)^m \mod (x^q-x)$ we are interested in coefficients of terms of the form $x^{i(q-1)}$ in the expansion of $f(x)^m$. But $\deg (f(x)^m)=nm$, and we have, since $nm \geq 6$,
\[ q-1=nm-2<nm<2nm-4=2(q-1). \]
So there are no terms of the form $x^{i(q-1)}$ for $i \geq 2$. 

We use the multinomial theorem to find the coefficient of $x^{q-1}=x^{nm-2}$. By \tref{thm:multinomial} we have
\[
f(x)^{m} = \sum_{\substack{k_1+\cdots+k_{n-2}\\ + k_n=m}} \begin{pmatrix} m \\ k_1,...,k_{n-2},k_n \end{pmatrix} a_1^{k_1} \cdots a_{n-2}^{k_{n-2}} \cdot x^{k_1+\cdots+(n-2)k_{n-2}+ n \cdot k_n}.
\]

To find the coefficient of $x^{q-1}=x^{nm-2}$ we must find all solutions over the nonnegative integers of the following system.
\begin{equation*}
\begin{cases}
k_1+\cdots+k_{n-2}+k_n=m, & (a)\\
k_1+2k_2+\cdots+(n-2)k_{n-2} +n\cdot k_n =nm-2. & (b)
\end{cases}
\end{equation*}
We observe that we must have $k_n=m-1$. For if $k_n=m$ then the \emph{LHS} of $(b)$ is immediately too large. On the other hand, if $k_n < m-1$ then the \emph{LHS} of $(b)$ is too small, for even if $k_{n-2}=m-k_n$ we have
\[ (n-2)(m-k_n)+n\cdot k_n = (n-2)m+2 k_n<n m-2.\]
So we must have $k_n=m-1$, in which case the system reduces to 
\begin{equation*}
\begin{cases}
k_1+\cdots+k_{n-2}=1, & (a)\\
k_1+2k_2+\cdots (n-2)k_{n-2} =n-2. & (b)
\end{cases}
\end{equation*}
The only solution is 
\[ k_1=\cdots=k_{n-3}=0, k_{n-2}=1,k_n=m-1.\]
Thus the coefficient of $x^{q-1}$ in $f(x)^m \mod (x^q-x)$ is $\frac{m!}{(m-1)!} a_{n-2}$, so we have
\[ m \cdot a_{n-2}=0. \]
Since we observed that $m \not \equiv 0 \mod p$ we must have $a_{n-2}=0$. 
\end{proof}

\section{Restrictions on $p$ and $q$}
In this section we determine necessary restrictions on $p$ and $q=p^r$ for \PP s of degree 6 to exist in \fqx. We first note that the affirmatively resolved Carlitz-Wan conjecture (\tref{thm:carlitz_wan}) gives us the upper bound $q \leq 6^4$. In fact we won't assume this theorem because Dickson's original classification in \cite{dickson1897II} claims to show this purely from Hermite's criterion. Indeed, we will find that all degree 6 \PP s of \fq\ satisfy $q \leq 27$. An interesting historical note is that Dickson's characterisation of degree 6 \PP s was used in partial proofs of the Carlitz conjecture \cite{hayes1967,wan1987,wan1990} before the general proof was found by Fried \emph{et al.} in \cite{fried1993}. 

Suppose that $f(x)$ is a degree 6 \PP\ of \fq, where $q$ is odd. To apply Hermite's criterion we will need to treat separately the different residue classes of $q$ modulo 6, for the term of degree $q-1$ in the expansion of $f(x)^t$ depends on the residue class of $q$. If $q=p^r$ then we write $q$ in the form
\[ q=6m+ \mu \text{, with } 0 \leq \mu \leq 5. \]
Now since $q$ is odd by assumption it is impossible that $\mu$ is even. Also, the case $\mu=1$ is impossible by \cref{cor:notequiv1}. So the two cases are $q=6m+3$, in which case $p=3$, and $q=6m+5$, in which case $p \equiv 5 \mod 3$ and $r$ is odd, as the following lemma shows.

\begin{lemma}\label{lem:p5mod6}
We have $p^r \equiv 5 \mod 6$ if and only if $p \equiv 5 \mod 6$ and $r$ is odd.
\end{lemma}

\begin{proof}
Since $5 \equiv -1 \mod 6$ the reverse implication is trivial. Suppose that $p^r \equiv 5 \mod 6$. Then $p^r \equiv 1 \mod 2$ and $p^r \equiv 2 \mod 3$. These conditions imply, respectively, that $p \equiv 1 \mod 2$, and $p \equiv 2 \mod 3$ with $r$ odd. Hence, $p \equiv 5 \mod 6$ and $r$ is odd. 
\end{proof}

\section{Degree 6 PPs of $\mathbb{F}_{6m+5}$}
The aim of this section is to classify all normalised \PP s of degree 6 over finite fields of the form $\mathbb{F}_{6m+5}$. By \dref{def:normalised} and \tref{thm:thirdhighestcoeffzero} such a polynomial has the general form
\begin{equation}\label{eq:normalised6}
f(x) = x^6+a_3 x^3 + a_2 x^2 + a_1 x.
\end{equation}
We remark that we are only interested in finite fields of order $q \geq 11$, because any degree 6 \PP\ of $\mathbb{F}_5$ can be reduced mod $x^5-x$ to a polynomial of degree $\leq 3$ (by \lref{lem:reduction} and \cref{cor:max_degree}). Thus, no \PP\ of $\mathbb{F}_5$ is a true degree 6 polynomial.

For any \PP\ $f(x)$ of the form \eref{eq:normalised6} we now use Hermite's criterion to derive a set of necessary equations in the coefficients $a_1,a_3,a_3$. Since $\deg{(f(x)^m)}=6m=q-5$ and $\deg{(f(x)^{m+1})}=6m+6=q+1$, we observe that $f(x)^{m+1}$ is the first power of $f(x)$ with degree $\geq q-1$. Hence $m+1$ is the first useful power to apply in Hermite's criterion. However, \tref{thm:thirdhighestcoeffzero} ensures that the polynomial \eref{eq:normalised6} always satisfies the power $f(x)^{m+1}$, so we begin by considering the next useful power, namely $f(x)^{m+2}$. We will require the following inequality
\begin{equation}\label{eq:q_inequality}
1 \leq m \leq q-10 \text{ for all } q \geq 11.
\end{equation}

\begin{lemma}\label{lem:threehermites}
Let
\[ f(x) = x^6+ a_3 x^3 + a_2 x^2 + a_1 x \in \mfqx \]
be a \PP\ of \fq, where $q=p^r=6m+5$. If $q\geq 11$ then
\begin{equation}\label{eq:firsthermite}
a_2^2 + 2a_1 a_3 = 0.
\end{equation}
If $q>11$ then 
\begin{equation}\label{eq:secondhermite}
36 a_1^2 a_2 - 15 a_2^2 a_3^2 - 10 a_1 a_3^3 =0.
\end{equation}
If $q>17$ then
\begin{equation}\label{eq:thirdhermite}
72 a_1^4 - 12 a_2^5 - 240 a_1 a_2^3 a_3 - 360 a_1^2 a_2 a_3^2 + 55 a_2^2 a_3^4 + 22 a_1 a_3^5=0.
\end{equation}
\end{lemma}

\begin{proof}
Note that $m+\frac{5-p}{6}$ and $m+\left( \frac{5-p}{6}+p \right)$ are consecutive multiples of $p$. Since $p \geq 5$ this implies that there are no multiples of $p$ lying strictly  between $m$ and $m+6$; in particular, the integers $m+2,m+3$ and $m+4$ are not divisible by $p$. We also have, by \eref{eq:q_inequality},
\[ 3 \leq m+2,m+3,m+4 \leq q-6. \]
So by \tref{thm:hermite}, the reductions of $f(x)^{m+2},f(x)^{m+3},f(x)^{m+4}$ modulo $x^q-x$ have degree $\leq q-2$. 

First consider the expansion of $f(x)^{m+2}$. We are interested in the coefficient of $x^{q-1}$ in $f(x)^{m+2}\mod (x^q-x)$, which means we must find coefficients of the terms $x^{i(q-1)}$ in $f(x)^{m+2}$. But the highest power of $x$ in $f(x)^{m+2}$ is $6(m+2)=q+7$, so if there were any terms in $x^{i(q-1)}$ with $i \geq 2$ then we would have 
\[ q+7 \geq 2(q-1), \]
which is equivalent to $q \leq 9$. Since $q \geq 11$ there are no such terms, so we only need to consider the coefficient of $x^{q-1}=x^{6m+4}$. 

By \tref{thm:multinomial} we have
\[
f(x)^{m+2} = \sum_{\substack{k_1+k_2+k_3 \\ +k_6=m+2}} \begin{pmatrix} m+2 \\ k_1,k_2,k_3,k_6 \end{pmatrix} a_1^{k_1}a_2^{k_2}a_3^{k_3} \cdot x^{ k_1+2k_2+3k_3+6k_6}.
\]
We must find all solutions over the nonnegative integers of the following system.
\begin{equation*}
\begin{cases}
k_1+k_2+k_3+k_6=m+2, &(a)\\
k_1+2k_2+3k_3+6k_6 =6m+4. & (b)
\end{cases}
\end{equation*}
We give an outline of this routine task. First note that $k_6 \leq m$, for otherwise $k_6 \geq m+1$, in which case $6 k_6 \geq 6m+6$ in contradiction to $(b)$. We must also have $k_6 \geq m$, for otherwise $k_6 \leq m-1$, in which case \emph{LHS} of $(b)$ can be at most, with $k_3=m+2-k_6$,
\[ 3(m+2-k_6)+6k_6=3m+6+3k_6 \leq 6m+3. \]
Hence $k_6=m$, and the system reduces to
\begin{equation*}
\begin{cases}
k_1+k_2+k_3=2, & (a)\\
k_1+2k_2+3k_3 =4, & (b)
\end{cases}
\end{equation*}
which is easily solvable over the finite domain of possibilities. The solutions are 
\[
 \begin{array}{c|c|c|c}
  k_1 & k_2 &k_3 & k_6 \\
  \hline
  	0&2&0&m\\
 	 1 &0&1&m
 \end{array}
\]
Hence, the coefficient of $x^{q-1}$ in $f(x)^{m+2} \mod (x^q-x)$ is 
\begin{align*}
\dfrac{(m+2)!}{2!m!} a_2^2 + \dfrac{(m+2)!}{m!}a_1 a_3 = \dfrac{(m+2)(m+1)}{2} a_2^2 + (m+2)(m+1)a_1 a_3.
\end{align*}
Equating this to zero (by \tref{thm:hermite}) and dividing by $(m+2)(m+1)/2 \neq 0$ we have 
\[ a_2^2 + 2a_1 a_3 =0. \]
	Now suppose that $q>11$, so that $q \geq 17$ since $q \equiv 5 \mod 6$. By a similar process we must find the coefficient of $x^{q-1}$ in $f(x)^{m+3} \mod (x^q-x)$. As before, we note that there are no terms in $f(x)^{m+3}$ of the form $x^{i(q-1)}, i\geq 2$. For the highest power of $x$ in $f(x)^{m+3}$ is $6m+18=q+13$, and $2(q-1)> q+13$ for all $q>15$. So we only need to find the coefficient of $x^{q-1}$.

To find the coefficient of $x^{q-1}=x^{6m+4}$ in $f(x)^{m+3}$ we must solve the following system 
\begin{equation*}
\begin{cases}
k_1+k_2+k_3+k_6=m+3, & (a)\\
k_1+2k_2+3k_3+6k_6 =6m+4. & (b)
\end{cases}
\end{equation*}
By similar reasoning to the previous case we deduce that $m-1 \leq k_6 \leq m$, and in both cases the system is easily solvable. The solutions are 
\[
 \begin{array}{c|c|c|c}
  k_1 & k_2 &k_3 & k_6 \\
  \hline
 	 2 &1&0&m\\
 	 0&2&2&m-1\\
	1&0&3&m-1
 \end{array}
\]
Hence, by \tref{thm:hermite}, we have the identity
\begin{align*}
0&= \dfrac{(m+3)!}{2\cdot m!} a_1^2 a_2 + \dfrac{(m+3)!}{4(m-1)!} a_2^2 a_3^2 + \dfrac{(m+3)!}{6(m-1)!} a_1 a_3^3  \\
&= \frac{(m+3)(m+2)(m+1)}{12} \left[ 6 a_1^2 a_2 +m(3 a_2^2 a_3^2 + 2a_1 a_3^3 )   \right]
\end{align*}
We may divide by $(m+3)(m+2)(m+1)/12 \neq 0$ and substitute $m=-5/6$ (since $m=(p^r-5)/6$). Simplifying, we have
\[ 36 a_1^2 a_2 - 15 a_2^2 a_3^2 - 10 a_1 a_3^3 =0. \]
Finally, suppose that $q>17$ and consider the coefficient of $x^{q-1}$ in $f(x)^{m+4}$ mod $(x^q-x)$. As before we deduce that there are no terms of the form $x^{i(q-1)}, i\geq 2$, so we only need to consider the coefficient of $x^{q-1}$ in $f(x)^{m+4}$. To find this coefficient we must solve
\begin{equation}\label{eq:thirdsystem}
\begin{cases}
k_1+k_2+k_3+k_6=m+4, & (a)\\
k_1+2k_2+3k_3+6k_6 =6m+4. & (b)
\end{cases}
\end{equation}
We have $m-2 \leq k_6 \leq m$, and the solutions are
\[
 \begin{array}{c|c|c|c}
  k_1 & k_2 &k_3 & k_6 \\
  \hline
	4&0&0&m\\
	0&5&0&m-1\\
	1&3&1&m-1\\
	2&1&2&m-1\\
	0&2&4&m-2\\
	1&0&5&m-2
 \end{array}
\]
We must perhaps address the fact that $k_6=m-2$ is not a valid solution to \eref{eq:thirdsystem} if $m <2$. However, since $q > 17$ and $m=(q-5)/6$ we have in fact that $m >2$, so that the given solutions are valid. Hence we have the identity
\[
\frac{(m+4)!}{m!} \frac{a_1^4}{4!}+\frac{(m+4)!}{(m-1)!} \left( \frac{a_2^5}{5!}+\frac{a_1 a_2^3 a_3}{6}+\frac{a_1^2 a_2 a_3^2}{4} \right) + \frac{(m+4)!}{(m-2)!}  \left( \frac{a_2^2 a_3^4}{2\cdot 4!}+\frac{a_1 a_3^5}{5!} \right) =0.
\]
Dividing by $(m+4)(m+3)(m+2)(m+1) \neq 0$, substituting $m=-5/6$ and simplifying we have
\[
\frac{a_1^4}{4!} - \frac{a_2^5}{6 \cdot 4!} -\frac{5 a_1 a_2^3 a_3}{36} - \frac{5 a_1^2 a_2 a_3^2}{24} +  \frac{55 a_2^2 a_3^4}{72 \cdot 4!} + \frac{11 a_1 a_3^5}{36 \cdot 4!}=0.
\]
Multiplying by $72\cdot 4!=2^6 \cdot 3^3 \neq 0$ we have
\[
72 a_1^4 - 12 a_2^5 - 240 a_1 a_2^3 a_3 - 360 a_1^2 a_2 a_3^2 + 55 a_2^2 a_3^4 + 22 a_1 a_3^5=0.
\]
\end{proof}

Armed with the equations derived in \lref{lem:threehermites} our next goal is to prove that there are no degree 6 \PP s of \fq\ if $q > 11$. The following important lemma shows that the linear term of a normalised degree 6 \PP\ of \fq\ is necessarily nonzero.
\begin{lemma}\label{lem:a1notzero}
Let
\[ f(x) = x^6+ a_3 x^3 + a_2 x^2 + a_1 x \in \mfqx \]
be a \PP\ of \fq, where $q=p^r=6m+5$. Then $a_1 \neq 0$.
\end{lemma}

\begin{proof}
If $a_1=0$ then \eref{eq:firsthermite} implies that $a_2=0$, so that $f(x)=x^6+a_3 x^3$. Consider the quadratic polynomial $g(x)=x^2+a_3 x \in \mfqx$. By normalisation $g(x)$ is a \PP\ of \fq\ if and only if the monomial $x^2$ is a \PP, which occurs precisely when $3 \mid q$ (\tref{thm:elementary}). Since $p \neq 3$, $g(x)$ is not a \PP\ of \fq, so neither is $g(x^3)=f(x)$. 
\end{proof}

We now show that there are no degree \PP s of \fq\ in the special case $q=17$. 
\begin{theorem}\label{thm:nopps17}
There are no degree 6 \PP s of $\mathbb{F}_{17}$.
\end{theorem}

\begin{proof}
Suppose that $f(x)$ is a degree 6 \PP\ of $\mathbb{F}_{17}$. By normalisation and \eref{eq:normalised6} we may express $f(x)$ in the form
\[ f(x) = x^6+a_3 x^3+a_2 x^2+a_1 x. \]
Reducing \eref{eq:firsthermite} and \eref{eq:secondhermite} modulo 17, we have
\[
\begin{cases}
a_2^2+2a_1 a_3=0, &(1)\\
2 a_1^2 a_2+2 a_2^2 a_3^2 +7 a_1 a_3^2=0. &(2).
\end{cases}
\]
We use Hermite's criterion to derive a third necessary equation in the coefficients of $f(x)$. By \tref{thm:hermite} the reduction of $f(x)^6$ modulo $x^{17}-x$ has degree $\leq 15$. Upon performing this expansion and equating the coefficient of $x^{16}$ to zero we have
\[ 
15 a_1^4 + 6 a_2 + 6 a_2^5 + a_1 a_2^3 a_3 + 10 a_1^2 a_2 a_3^2 + 15 a_2^2 a_3^4 + 6 a_1 a_3^5 =0. \quad(3)
\]
Since $a_1 \neq 0$ (\lref{lem:a1notzero}) we have by (1) that $a_3=-a_2^2/(2a_1)$. Substituting this into (2) and (3) and multiplying each by a suitable power of $a_1$ we get
\[
\begin{cases}
2 a_1^4 a_2 + 6 a_2^6=0, &(2')\\
15 a_1^8 + 6 a_1^4 a_2 + 8 a_1^4 a_2^5 + 5 a_2^{10}=0. &(3')
\end{cases}
\]
If $a_2=0$ then $(3')$ implies that $a_1=0$ in contradiction to \lref{lem:a1notzero}, so we may assume that $a_2 \neq 0$. Dividing $(2')$ by $a_2$ and simplifying, we have $a_2^5=11a_1^4$. Subsitituing this into $(3')$ we get $a_2 = a_1^4$. Hence we have
\[
\begin{cases}
a_2^5=11a_1^4, &(2'')\\
a_2 = a_1^4. &(3'').
\end{cases}
\]
But dividing $(2'')$ by $(3'')$ gives $a_2^4=11$, and 11 has no fourth root in $\mathbb{F}_{17}$, a contradiction.
\end{proof}

We are now able to prove the more general result that there are no degree 6 \PP s of \fq\ when $q>11$.
\begin{theorem}\label{thm:nopps>11}
Let $q=6m+5>11$. Then there are no degree 6 \PP s of \fq.
\end{theorem}

\begin{proof}
Let $f(x)$ be a degree 6 \PP\ of \fq, where $q=p^r=6m+5 >11$. Since the $q=17$ case was considered in \tref{thm:nopps17} we may assume that $q>17$.

Moreover we may assume that $f$ is normalised, so by \eref{eq:normalised6} and \lref{lem:threehermites} we have
\[ f(x) = x^6+ a_3 x^3 + a_2 x^2 + a_1 x, \]
where
\begin{equation*}
\begin{cases}
a_2^2 + 2a_1 a_3 = 0, &(1)\\
36 a_1^2 a_2 - 15 a_2^2 a_3^2 - 10 a_1 a_3^3 =0, &(2) \\
72 a_1^4 - 12 a_2^5 - 240 a_1 a_2^3 a_3 - 360 a_1^2 a_2 a_3^2 + 55 a_2^2 a_3^4 + 22 a_1 a_3^5=0. &(3)
\end{cases}
\end{equation*}
Note also that $a_1\neq 0$ by \lref{lem:a1notzero}.

If $a_2=0$ then $a_3=0$ by (1), but then by (3) we have $a_1=0$, a contradiction. So we may assume that $a_2 \neq 0$.

Now by (1) we have $a_3=-a_2^2/(2a_1)$. Substituting this into (2) and (3) and dividing by $a_2\neq 0$ where necessary, we have
\begin{equation*}
\begin{cases}
5 a_2^5=72 a_1^4, &(2')\\
288 a_1^8 + 72 a_1^4 a_2^5 + 11 a_2^{10}=0. &(3')
\end{cases}
\end{equation*}
If $p=5$ then $(2')$ implies that $a_1=0$, a contradiction. If $p \neq 5$ then substituting $a_2^5=72a_1^4/5$ into $(3')$ gives
\[ \frac{90144}{25} a_1^8 =0, \]
and since $90144=2^5\cdot3^2\cdot 313$ is a product of primes $\not\equiv 5 \mod 6$, we have $a_1=0$, a contradiction.
\end{proof}

Our results so far have reduced the characterisation of degree 6 \PP s of $\mathbb{F}_{6m+5}$ to the case $6m+5=11$. In contrast to the fields of higher order there do exist degree 6 \PP s of $\mathbb{F}_{11}$, and the following theorem gives their complete characterisation. 
\begin{theorem}
The following is the complete list of normalised degree 6 \PP s of $\mathbb{F}_{11}$:
\[
\begin{matrix}
x^6 \pm 2x, \\
x^6 \pm 4x, \\
x^6 \pm a^2x^3+ax^2 \pm 5x \textnormal{ ($a$ a nonzero square)}, \\
x^6 \pm 4 a^2 x^3+ax^2 \pm 4x \textnormal{ ($a$ a nonsquare)}.
\end{matrix}
\]
\end{theorem}

\begin{proof}
By \eref{eq:normalised6}, let
\[ f(x) = x^6+a_3 x^3+a_2 x^2+a_1 x \in \mathbb{F}_{11}[x] \]
be a normalised \PP\ of $\mathbb{F}_{11}$. Then by \tref{thm:hermite} the reductions of $f(x)^3$, $f(x)^4$ and $f(x)^5$ modulo $x^{11}-x$ must have degree $\leq 10$. Performing these expansions (routine calculations omitted) and equating the coefficient of $x^{10}$ to zero in each case, we get the necessary conditions
\[
\begin{cases}
a_2^2 + 2 a_1 a_3=0, &(1)\\
4 a_2 + a_1^2 a_2 + 6 a_2^2 a_3^2 + 4 a_1 a_3^3=0, &(2)\\
1 + 10 a_1^2 + 5 a_1^4 + a_2^5 + 9 a_1 a_2^3 a_3 + 8 a_2 a_3^2 + 8 a_1^2 a_2 a_3^2=0. &(3)
\end{cases}
\]
Since $a_1\neq 0$ (\lref{lem:a1notzero}), we may express (1) as $a_3=-a_2^2/(2 a_1)$. Substituting this into (2) and (3) and simplifying gives
\[
\begin{cases}
a_3=-a_2^2/(2a_1), &(1') \\
4 a_1^2 a_2 + a_1^4 a_2 + a_2^6=0, &(2')\\
2 a_1^2 + 9 a_1^4 + 10 a_1^6 + 4 a_2^5 + 8 a_1^2 a_2^5=0. &(3')
\end{cases}
\]
If $a_2=0$ then $a_3=0$ by $(1')$, and $(3')$ reduces to 
\[ a_1^4+ 2 a_1^2+9 =0. \]
By the quadratic formula we then have $a_1^2 = 4$ or $5$, so that $a_1= \pm 2$ or $\pm 4$. Thus we have the following candidates for \PP s:
\[ x^6 \pm 2x, x^6 \pm 4x. \]
If $a_2 \neq 0$ then we may divide $(2')$ by $a_2$ and rearrange to get $a_2^5= 10 a_1^2(a_1^2+4)$. Substituting this into $(3')$ we have
\[ a_1^4 + 3a_1^2+4. \]
By the quadratic formula we have $a_1^2 = 3$ or $5$ so that $a_1= \pm 4$ or $\pm 5$.

If $a_1=\pm 4$ then we have $a_2^5= 10 \cdot 4^2 \cdot (4^2+4)=-1$. By \eref{eq:squares}, $a_2$ is a nonsquare in $\mathbb{F}_{11}$. We then have $a_3=-a_2^2/(\pm 8)=\pm4a_2^2$. Denoting $a_2$ by $a$ this gives us the family of candidate polynomials
\[ x^6 \pm 4a^2x^3 + ax^2 \pm 4x \text{ ($a$ a nonsquare)}. \]
If $a_1=\pm 5$ then we have $a_2^5=10 \cdot 5^2 \cdot (5^2+4)=1$. By \eref{eq:squares}, $a_2$ is a nonzero square in $\mathbb{F}_{11}$. We then have $a_3=-a_2^2/(\pm 10)=\pm a_2^2$. Denoting $a_2$ by $a$ this gives us the family of candidate polynomials
\[ x^6 \pm a^2x^3 + ax^2 \pm 5x \text{ ($a$ a nonzero square)}. \]
Routine checking shows that all of the polynomials given satisfy the remaining powers in \tref{thm:hermite}, so they are indeed \PP s.
\end{proof}

\section{Degree 6 PPs of $\mathbb{F}_{6m+3}$}\label{sec:3k}
If $q=p^r=6m+3$ then $p=3$; the goal of this section is to characterise all degree 6 \PP s over finite fields of the form $\mathbb{F}_{3^r}$. Note that we are only interested in $r \geq 2$, because any degree 6 \PP\ of $\mathbb{F}_3$ may be reduced mod $x^3-x$ to a linear polynomial (by \lref{lem:reduction} and \cref{cor:max_degree}). So \PP s of $\mathbb{F}_3$ cannot be true degree 6 polynomials. 

Although normalisation in the sense of \dref{def:normalised} only allows us to restrict the constant term and the coefficient of $x^6$, the following lemma uses a linear transformation to additionally remove the coefficient of either $x^5$ or $x^4$. It shows that if we can characterise all degree 6 \PP s with at most one $x^5$ or $x^4$ term, then via linear transformations we can obtain the full list of \PP s.
\begin{lemma}\label{lem:a5ora4=0}
Let 
\[ f(x) = x^6+a_5 x^5 + a_4 x^4 + a_3 x^3 + a_2 x^2 + a_1 x \]
be a normalised \PP\ of degree 6 in $\mathbb{F}_{3^r}$. If $a_5 \neq 0$ then by a transformation of the form $f(x+b)+c$ we can remove the $x^4$ term.
\end{lemma}

\begin{proof}
Expanding $f(x+b)+c$ in $\mathbb{F}_{3^r}$ we have 
\begin{multline*}
f(x+b)+c=   x^6 + a_5 x^5 + (a_4 + 
    2 a_5 b) x^4 + (a_3 + a_4 b + a_5 b^2 + 2 b^3) x^3 \\+(a_2 + 
    a_5 b^3) x^2 + (a_1 + 2 a_2 b + a_4 b^3 + 2 a_5 b^4) x \\+ (a_1 b + a_2 b^2 + a_3 b^3 + a_4 b^4+ a_5 b^5 + b^6 + c). 
\end{multline*}
If $a_5 \neq 0$ then set $b=a_4/a_5$ to remove the $x^4$ term and set $c=- (a_1 b + a_2 b^2 + a_3 b^3 + a_4 b^4+ a_5 b^5 + b^6)$ to remove the constant term.
\end{proof}

We split our characterisation of degree 6 \PP s into two cases; first the special case $q=3^2$, then the general case $q>3^2$.
\subsection{Degree 6 PPs of $\mathbb{F}_{3^2}$}
In this section $2^{1/2}$ is a symbol for \emph{either} solution of $x^2-2=0$ in $\mathbb{F}_{3^2}$. \\[0.3cm] We first consider the case $a_5=0$.

\begin{theorem}\label{thm:f9a5zero}
The complete list of \PP s of $\mathbb{F}_{3^2}$ of the form 
\[
f(x)=x^6+a_4 x^4 +a_3x^3+a_2x^2+a_1x
\]
is given by 
\begin{gather*}
x^6+a^2 x^4 + a^7 b x^3 + a^4 x^2 + a(2b+1)x, \\
a\neq 0, b\in \{0,1,2^{1/2},1+2^{1/2} \}.
\end{gather*}
\end{theorem}

\begin{proof}
Let $f(x)$ be a \PP\ of $\mathbb{F}_{3^2}$. Then by Hermite's criterion (\tref{thm:hermite}) the reductions of $f(x)^2,f(x)^4$ and $f(x)^5$ modulo $x^9-x$ have degree $\leq 7$. Performing these expansions (routine calculations omitted) and equating the coefficient of $x^8$ to zero in each case, we get the necessary conditions
\[
\begin{cases}
a_2 = a_4^2, &(1)\\
1 + a_2^4 + a_4^4=0, &(2)\\
a_1^2 a_2^3 + 2 a_1^3 a_2 a_3 + a_3^2 + 2 a_1 a_3^3 + 2 a_1^4 a_4 + \\ \quad 2 a_2 a_4 +2 a_2^3 a_4 + 2 a_3^4 a_4  + a_4^3 + a_2^2 a_4^3 + 2 a_1 a_3 a_4^3=0. &(3)
\end{cases}
\]
From (1) and (2) we have $1+a_4^4+a_4^8=0$. In particular this shows that $a_4\neq0$, so we may let $a_4^8=1$ and reduce the equation to $a_4^4=1$. By \eref{eq:squares}, $a_4$ is a nonzero square in $\mathbb{F}_{3^2}$. 

\noindent Substituting (1) into (3) we have
\begin{align*}
0 &= a_3^2 + 2 a_1 a_3^3 + 2 a_1^4 a_4 + 2 a_3^4 a_4 + 2 a_1^3 a_3 a_4^2 + 2 a_1 a_3 a_4^3 + a_1^2 a_4^6\\
	&= (a_1^2 a_4^6 + 2 a_1a_3a_4^3 +a_3^2) + 2 (a_1 + a_3 a_4) (a_1^3 a_4+a_3^3).
\end{align*}
Multiplying by $a_4^3 \neq 0$ and simplifying via $a_4^4=1$ we have
\begin{align*}
0 &= a_4(a_1^2+2a_1a_3a_4 + a_3^2a_4^2) + 2(a_1 + a_3a_4) (a_1^3+a_3^3a_4^3)\\
	&= a_4(a_1+a_3a_4)^2 +2(a_1 +a_3a_4)^4\\
	&=(a_1+a_3a_4)^2(a_4+2(a_1+a_3a_4)^2).
\end{align*}
Hence either $a_1=2a_3a_4$ or $a_4=(a_1+a_3a_4)^2$. But if $a_1=2a_3a_4$ then the polynomial 
\[ x^6+a_4 x^4 + a_3 x^3+a_4^2 x^2+ 2a_3 a_4 x \]
has roots at 0 and $a_4^{1/2} \neq 0$, thus failing to be injective. So we must have $a_1=2a_3a_4\pm a_4^{1/2}$.

\noindent Thus (1)-(3) are satisfied precisely when:
\[
\begin{cases}
a_4 \text{ is a nonzero square},\\
a_2 = a_4^2,\\
a_1=2a_3a_4\pm a_4^{1/2}.
\end{cases}
\]
It is convenient to give the following parametrisation, where $a$ is an arbitrary nonzero element of $\mathbb{F}_{3^2}$:
\[
\begin{cases}
a_4= a^2,\\
a_2 = a^4,\\
a_1=2a_3a^2+a.
\end{cases}
\]
We have chosen the values $a_1,a_2,a_4$ to satisfy the powers $2,4,5$ in Hermite's criterion. Indeed, it happens that these choices also ensure that the power 7 is satisfied. Hence $f(x)$ is a \PP\ if and only if the reduction of $f(x)^8$ modulo $x^9-x$ is monic of degree 8.  Now we have shown that $f(x)$ must be of the form (with $a\neq0$)
\begin{equation*}
f(x)=x^6+a^2x^4+a_3x^3+a^4x^2+a(2a a_3+1)x.
\end{equation*}
Expanding $f(x)^8 \mod (x^9-x)$ and equating the coefficient of $x^8$ to 1, we have
\[ 1 + a a_3 + a^2 a_3^2 + a^3 a_3^3 + 2 a^4 a_3^4 + a^6 a_3^6 =1. \]
After factorisation this condition becomes 
\[a a_3 (a a_3+2) (a^2 a_3^2-2) ((a a_3 +2)^2-2)= 0,\]
which is satisfied whenever $a a_3=0,1,2^{1/2}$ or $1 + 2^{1/2}$. Equivalently, $a_3=0$, $a^{-1}$, $2^{1/2}a^{-1}$ or $(1 + 2^{1/2})a^{-1}$. Using $a^{-1}=a^7$ and simplifying gives us the following family of \PP s:
\begin{gather*}
x^6+a^2 x^4 + a^7 b x^3 + a^4 x^2 + a(2b+1)x, \\
a\neq 0, b\in \{0,1,2^{1/2},1+2^{1/2}\} .
\end{gather*}
\end{proof}
We compare this result to other characterisations of this case in the literature. The family of \PP s given above is equivalent to the original family proposed by Dickson in \cite{dickson1897II}, however we suggest that our parametrisation is much cleaner than his family given by
\begin{gather*}
x^6+ax^4+bx^3+a^2x^2+(2ab \pm a^{5/2})x \\
a \textit{ square}, a \neq 0; b=0, \pm 2^{1/2} a^{3/2}, \pm a^{3/2}, \text{ or } \pm (2^{1/2} +1)a^{3/2}. \\
\textit{ The signs of $b$ to correspond to that of } \pm a^{5/2}.
\end{gather*}
On the other hand, the characterisation given in \cite[Theorem~3.14]{evans1992} is incorrect, for, in particular, it suggests that the coefficient of $x^4$ must be a fourth power.

We now consider the case $a_5 \neq 0$. In light of \lref{lem:a5ora4=0} we may assume that $a_4=0$.
\begin{theorem}\label{thm:f9a4zero}
The complete list of \PP s of $\mathbb{F}_{3^2}$ of the form 
\[
f(x)=x^6+a_5 x^5 +a_3x^3+a_2x^2+a_1x
\]
with $a_5 \neq 0$ is given by
\begin{gather*}
x^6+a x^5 + a^3 x^3 + 2 a^4 x^2 +2a^5 x \quad (a\neq 0), \\
x^6 + a x^5 + \varphi a^3 x^3 + 2 \varphi a^4 x^2 + 2^{1/2} a^5 x \quad (a\neq0,\varphi=\pm(1-2^{1/2})), \\
x^6+a x^5+2 a^3 x^3+a^4 x^2 +  (2+2^{1/2}) a^5 x \quad (a \neq 0).
\end{gather*}
\end{theorem}

\begin{proof}
If $f(x)$ is a \PP\ of $\mathbb{F}_{3^2}$ then by \tref{thm:hermite} the reductions of $f(x)^2$, $f(x)^4$ and $f(x)^5$ modulo $x^9-x$ have degree $\leq 7$. Performing these expansions and equating the coefficient of $x^8$ to zero in each case we have
\[
\begin{cases}
a_2 = 2 a_3 a_5, &(1)\\
1 + a_2^4 + a_1^3 a_5 + a_1 a_5^3=0, &(2)\\
a_1^2 a_2^3 + 2 a_1^3 a_2 a_3 + a_3^2 + 2 a_1 a_3^3 + 2 a_1 a_5 + 2 a_2 a_3^3 a_5 + a_2^3 a_5^2 + 2 a_3 a_5^3=0. &(3)
\end{cases}
\]
First we show that $a_1 \neq 0$ and $a_2 \neq 0$. If $a_2 =0$ then $a_3=0$ by (1), from which it follows from (3) that $a_1=0$, in contradiction to (2). If $a_1=0$ then substituting (1) into (3) we have
\[ a_3^2 + a_3^4 a_5^2 + 2 a_3 a_5^3 + 2 a_3^3 a_5^5=0.\] 
Multiplying by $a_5^6 \neq 0$ and simplifying via $a_5^8=1$ we have
\[ a_3^4 + 2 a_3 a_5 + 2 a_3^3 a_5^3 + a_3^2 a_5^6 = a_3 (a_3 + 2 a_5^3) (a_3^2 + a_5^6) =0. \]
But each of $a_3=0$, $a_3=a_5^3$ and $a_3=2^{1/2} a_5^3$ lead to a contradiction in (2). So we must have $a_1 \neq 0$ and $a_2 \neq 0$. 

\noindent Squaring $(2)$ we have
\begin{align*} 
0 &= 1 + 2 a_2^4 + a_2^8 + 2 a_1^3 a_5 + 2 a_1 a_5^3  + 2 a_1^3 a_2^4 a_5 + 2 a_1 a_2^4 a_5^3+ a_1^6 a_5^2  + \\
	&\quad \, \, \,2 a_1^4 a_5^4 + a_1^2 a_5^6\\
	&= 2 + 2 a_2^4 + a_2^8 + 2( a_1^3 a_5 + a_1 a_5^3+1)  +a_2^4( 2 a_1^3 a_5+ 2 a_1 a_5^3) + a_1^6 a_5^2  +\\
 &\quad \, \, \,2 a_1^4 a_5^4 +  a_1^2 a_5^6\\
	&= 2 + 2 a_2^4 + a_2^8 + a_2^4  +a_2^4( 1+a_2^4) + a_1^6 a_5^2  + 2 a_1^4 a_5^4 + a_1^2 a_5^6\\
	&= 2 + a_2^4 + 2 a_2^8 + a_1^6 a_5^2 + 2 a_1^4 a_5^4 + a_1^2 a_5^6.
\end{align*} 
But $a_2^8=1$ since $a_2 \neq 0$, so we have
\begin{align*} 
0 &= (1 + a_2^4) + a_1^6 a_5^2 + 2 a_1^4 a_5^4 + a_1^2 a_5^6\\
	&=2 a_1^3 a_5+2 a_1 a_5^3 + a_1^6 a_5^2 + 2 a_1^4 a_5^4 + a_1^2 a_5^6.
\end{align*}
Dividing by $a_1 a_5 \neq 0$ and writing $a_1=\eta a_5^5$ we have
\begin{align*} 
0 &= 2 a_1^2 + a_1^5 a_5 + 2 a_5^2 + 2 a_1^3 a_5^3 + a_1 a_5^5 \\
	&= 2 + a_5^8 \eta + 2 a_5^8 \eta^2 + 2 a_5^{16} \eta^3 + a_5^{24} \eta^5\\
	&= \eta^5 + 2\eta^3 +2 \eta^2 + \eta +2\\
	&=(\eta+1)(\eta ^2+1)((\eta+1) ^2+1).
\end{align*}
Hence $\eta=2,2^{1/2}$ or $2+2^{1/2}$.

If $\eta=2$ then $a_1=2a_5^5$, and (1)-(3) reduce to 
\[
\begin{cases}
a_2 = 2 a_3 a_5, &(1)\\
a_2^4 =1, &(2)\\
a_3^4 a_5^2 + 2 a_3 a_5^3 + 2 a_3^3 a_5^5 + a_5^6=a_5^2(a_3-a_5^3)^4=0. &(3)
\end{cases}
\]
(1)-(3) are satisfied precisely when $a_3=a_5^3$ and $a_2=2 a_5^4$. Then one may check that for any $a_5 \neq 0$ the powers 7 and 8 in Hermite's criterion are also satisfied, so we have the following family of \PP s
\[
x^6+a x^5 + a^3 x^3 + 2 a^4 x^2 +2a^5 x \quad (a\neq 0).
\]

Now suppose that $\eta=2^{1/2}$. Note that although $2^{1/2}$ may refer to \emph{either} square root of 2 in $\mathbb{F}_{3^2}$ we assume that the particular choice is fixed. Then $a_1=2^{1/2} a_5^5$, and (1)-(3) reduce to 
\[
\begin{cases}
a_2 = 2 a_3 a_5, &(1)\\
a_2^4 =2, &(2)\\
(1- 2^{1/2}) a_3^2 + a_3^4 a_5^2 + 2 a_3 a_5^3 - 2^{1/2} a_3^3 a_5^5 - 2^{1/2} a_5^6=0. &(3)
\end{cases}
\]
Multiplying (3) by $a_5^2$ and letting $a_3^4 a_5^4=a_2^4=2$ and $a_3=\varphi a_5^3$ we have
\[ - 2^{1/2} \varphi^3+(1 -2^{1/2}) \varphi^2+2\varphi-(1 +2^{1/2})=0. \]
The roots of this polynomial are $\varphi=-1+2^{1/2},1-2^{1/2},-1-2^{1/2}$. Letting $a_5=a$ we have reduced to the following candidates
\begin{gather*}
x^6 + a x^5 + \varphi a^3 x^3 + 2 \varphi a^4 x^2 + 2^{1/2} a^5 x,\\
a\neq0,\varphi \in \{ \pm (1-2^{1/2}),-1-2^{1/2} \}.
\end{gather*}
If $\varphi=-1-2^{1/2}$ then this polynomial has roots 0 and $- 2^{1/2} a \neq 0$, thus failing to be injective. However if $\varphi=\pm(1-2^{1/2})$ then one may verify that (1)-(3) as well as the powers 7 and 8 in Hermite's criterion are satisfied. This gives us the family
\begin{gather*}
x^6 + a x^5 + \varphi a^3 x^3 + 2 \varphi a^4 x^2 + 2^{1/2} a^5 x,\\
a\neq0,\varphi = \pm (1-2^{1/2}).
\end{gather*}

Finally, if $\eta=2+2^{1/2}$ then $a_1=(2+2^{1/2}) a_5^5$, and (1)-(3) reduce to 
\[
\begin{cases}
a_2 = 2 a_3 a_5, &(1)\\
a_2^4 =1, &(2)\\
0=(1 - 2^{1/2}) a_5^6+2^{1/2} a_3^3 a_5^5 +2 a_3 a_5^3+a_3^4 a_5^2- 2^{1/2} a_3^2. &(3)
\end{cases}
\]
Multiplying (3) by $a_5^2$ and letting $a_3^4 a_5^4=a_2^4=1$ and $a_3=\varphi a_5^3$ we have
\[ 2^{1/2} \varphi^3-2^{1/2} \varphi^2-\varphi -(1+2^{1/2})=0. \]
The only root is $\varphi=2$, and the resulting family satisfies (1)-(3) as well as the powers 7 and 8 in Hermite's criterion:
\[
x^6+a x^5+2 a^3 x^3+a^4 x^2 +  (2+2^{1/2}) a^5 x \quad (a \neq 0).
\]
\end{proof}

\subsection{Degree 6 PPs of $\mathbb{F}_{3^r}$, $r>2$}\label{sec:ppsF9}
We now address the more general case of classifying degree 6 \PP s of $\mathbb{F}_{3^r}$ for all $r>2$. We will require the 3-adic expansion of $m=(3^{r-1}-1)/2$:
\begin{equation}\label{eq:3adic_m}
 m=1+3+\cdots +3^{r-3} + 3^{r-2}.
\end{equation}

\begin{lemma}\label{lem:m_inequality3}
Let $q=3^r=6m+3$ where $r > 2$. Then $1 < m <q-8$ and $m \equiv 1 \mod 3$. 
\end{lemma}
\begin{proof}
We have $m=(q-3)/6 < q-8$ if and only if $q > 9$, which is true since $r > 2$. Similarly, $(q-3)/6 > 1$ if and only if $q > 9$. It is immediate from \eref{eq:3adic_m} that $m\equiv1 \mod 3$.
\end{proof}

As with the case $q=6m+5$ we use Hermite's criterion to derive necessary equations in the coefficients of a normalised \PP\ $f(x)$. Again we observe that $f(x)^{m+1}$ is the first power of $f$ with degree exceeding $q-1$. Hence $m+1$ is the first useful power to apply in Hermite's criterion.

In the next theorem we reduce the characterisation problem to the case $a_5 \neq 0$, but we will need a lemma first. Using the powers $m+1,m+4$ and $3m+1$ in Hermite's criterion we determine a set of necessary equations in the coefficients of a degree 6 \PP\ satisfying $a_5=0$. The reward for this long and tricky lemma is that the equations will be proven inconsistent in $\mathbb{F}_{3^r}$, thus showing that the case $a_5=0$ is empty. 
\begin{lemma}\label{lem:a5not0equations}
Suppose that
\[ 
f(x)=x^6+a_4 x^4 + a_3 x^3 + a_2 x^2 + a_1x \in \mathbb{F}_{q}[x]
\]
is a \PP\ of $\mathbb{F}_{q}$, where $q=3^r=6m+3$. If $r>2$ then we have
\[
\begin{cases}
a_2 = a_4^2, &(1)\\
2 a_1^2 a_2^3 + a_1^3 a_2 a_3 + 2 a_2 a_3^6 + a_1^4 a_4 +  \\
\quad  a_3^6 a_4^2 + a_2 a_3^4 a_4^3 + a_1 a_3^3 a_4^4 + 2 a_2^3 a_4^5 + a_1^2 a_4^6=0,&(2)\\
1 + a_4^{3m+1} + a_2^{3m+1} =0. &(3)
\end{cases}
\]
\end{lemma}

\begin{proof}
Suppose that $f(x)$ is a \PP\ of $\mathbb{F}_{q}$, where $q=3^r$ and $r>2$. By \lref{lem:m_inequality3} we have that $m+1$ and $m+4$  are nonzero mod 3, and 
\[ 2 < m+1,m+4 < q-4. \]
We also have $3m+1\equiv 1\not\equiv 0 \mod 3$ and  that
\[ 1 \leq 3m+1 \leq 6m+1=q-2 \text{ for all } m \geq 0. \]
So by \tref{thm:hermite} the reductions of $f(x)^{m+1},f(x)^{m+4}$ and $f(x)^{3m+1}$ modulo $x^q-x$ each have degree $\leq q-2$. 

We use the multinomial theorem to expand these powers. By \tref{thm:multinomial} we have, for any positive integer $t$,
\begin{equation}\label{eq:multinomial3}
f(x)^{t} = \sum_{\substack{k_1+k_2+k_3 \\ +k_4+k_6=t}} \begin{pmatrix} t \\ k_1,k_2,k_3,k_4,k_6 \end{pmatrix} a_1^{k_1}a_2^{k_2}a_3^{k_3}a_4^{k_4} \cdot x^{ k_1+2k_2+3k_3+4k_4+6k_6}.
\end{equation}
First consider the expansions of $f(x)^{m+1}$ and $f(x)^{m+4}$. We are interested in the terms $x^{i(q-1)}$, but since $2(q-1)>6(m+1)=q+3$ for all $q>5$ and $2(q-1)>6(m+4)=q+21$ for all $q>23$ there are no terms of the form $x^{i(q-1)},i\geq2$. Hence, in each case, we only need to find the coefficient of $x^{q-1}=x^{6m+2}$. For the power $f(x)^{m+1}$ this amounts to solving the system 
\begin{equation*}
\begin{cases}
k_1+k_2+k_3+k_4+k_6=m+1, &(a)\\
k_1+2k_2+3k_3+4k_4+6k_6 =6m+2, & (b)
\end{cases}
\end{equation*}
for which the solutions are
\begin{equation}\label{eq:multinomalsolutions1}
 \begin{array}{c|c|c|c|c}
  k_1 & k_2 &k_3 &k_4& k_6 \\
  \hline
	0&1&0&0&m\\
	0&0&0&2&m-1
 \end{array}.
\end{equation}
Similarly, to find the coefficient of $x^{6m+2}$ in $f(x)^{m+4}$ we must solve the system
\begin{equation*}
\begin{cases}
k_1+k_2+k_3+k_4+k_6=m+4, & (a)\\
k_1+2k_2+3k_3+4k_4+6k_6 =6m+2. & (b)
\end{cases}
\end{equation*}
There are in fact 34 solutions. A partial list is
\begin{equation}\label{eq:multinomalsolutions2}
 \begin{array}{c|c|c|c|c}
  k_1 & k_2 &k_3 &k_4& k_6 \\
  \hline
	4&0&0&1&m-1\\
	\vdots&\vdots&\vdots&\vdots&\vdots \\
	3&0&1&2&m-2\\
	2&2&0&2&m-2\\
	2&1&2&1&m-2\\
	\vdots&\vdots&\vdots&\vdots&\vdots \\
	0&0&4&5&m-5\\
	0&0&2&8&m-6\\
	0&0&0&11&m-7
 \end{array}.
\end{equation}
Now \eref{eq:multinomalsolutions1} and \eref{eq:multinomalsolutions2} determine terms in $f(x)^{m+1}$ and $f(x)^{m+4}$, respectively, of the form
\[ \begin{pmatrix} m+1 \\ k_1,k_2,k_3,k_4,k_6 \end{pmatrix} a_1^{k_1}a_2^{k_2}a_3^{k_3}a_4^{k_4} x^{q-1} \text{ and } \begin{pmatrix} m+4 \\ k_1,k_2,k_3,k_4,k_6 \end{pmatrix}a_1^{k_1}a_2^{k_2}a_3^{k_3}a_4^{k_4} x^{q-1}. \]
For each solution in \eref{eq:multinomalsolutions1} and \eref{eq:multinomalsolutions2} we apply \tref{thm:lucas} to calculate the corresponding multinomial coefficient mod 3. 
We give two examples of this below. 

Consider the first solution in \eref{eq:multinomalsolutions2}. 
Using the 3-adic expansion of $m$ from \eref{eq:3adic_m} we have
\begin{align*}
m+4 &= 2 + 2 \cdot 3 + 1 \cdot 3^2 + \cdots + 1 \cdot 3^{r-3} + 1 \cdot 3^{r-2} \\
m-1 &= 0 + 1 \cdot 3 + 1 \cdot 3^2 + \cdots + 1 \cdot 3^{r-3} + 1 \cdot 3^{r-2}\\
4 &= 1 + 1 \cdot 3 + 0 \cdot 3^2 + \cdots + 0 \cdot 3^{r-3} + 0 \cdot 3^{r-2} \\
1 &= 1 + 0 \cdot 3 + 0 \cdot 3^2 + \cdots + 0 \cdot 3^{r-3} + 0 \cdot 3^{r-2}
\end{align*}
We observe that there are no `carries' in the sum $(m+4)=(m-1)+4+1$, so by \tref{thm:lucas} the multinomial coefficient $\left( \begin{smallmatrix} m+4 \\ 4,0,0,1,m-1 \end{smallmatrix} \right)$ is nonzero mod 3, and we have
\begin{align*}
\begin{pmatrix} m+4 \\ 4,0,0,1,m-1 \end{pmatrix} &\equiv \begin{pmatrix} 2 \\ 1,0,0,1,0 \end{pmatrix} \begin{pmatrix} 2 \\ 1,0,0,1,1 \end{pmatrix} \begin{pmatrix} 1 \\ 0,0,0,0,1 \end{pmatrix} \cdots \\
	&\quad \, \begin{pmatrix} 1 \\ 0,0,0,0,1 \end{pmatrix} \mod 3 \\
	&\equiv 2 \cdot 2 \mod 3\\
	&\equiv 1 \mod 3.
\end{align*}
Thus we get a term of the form $a_1^4 a_4 x^{q-1}$ in the expansion of $f(x)^{m+4}$ mod $(x^q-x)$. 

On the other hand, for the second solution listed in \eref{eq:multinomalsolutions2} we have
\begin{align*}
m+4 &= 2 + 2 \cdot 3 + 1 \cdot 3^2 + \cdots + 1 \cdot 3^{r-3} + 1 \cdot 3^{r-2} \\
m-2 &= 2 + 0 \cdot 3 + 1 \cdot 3^2 + \cdots + 1 \cdot 3^{r-3} + 1 \cdot 3^{r-2}\\
3 &= 0 + 1 \cdot 3 + 0 \cdot 3^2 + \cdots + 0 \cdot 3^{r-3} + 0 \cdot 3^{r-2} \\
1 &= 1 + 0 \cdot 3 + 0 \cdot 3^2 + \cdots + 0 \cdot 3^{r-3} + 0 \cdot 3^{r-2} \\
2 &= 2 + 0 \cdot 3 + 0 \cdot 3^2 + \cdots + 0 \cdot 3^{r-3} + 0 \cdot 3^{r-2}
\end{align*}
In this case there \emph{is} a carry in the sum $(m+4)=(m-2)+3+1+2$, so by \tref{thm:lucas} the multinomial coefficient $\left( \begin{smallmatrix} m+4 \\ 3,0,1,2,m-2 \end{smallmatrix} \right)$ is zero mod 3.

By similar computations (this process can be automated) we calculate the remaining binomial coefficients mod 3, and hence the coefficients of $x^{q-1}$ in $f(x)^{m+1}$ and $f(x)^{m+4}$ mod $x^q-x$. Equating them to zero we have, respectively,
\[
\begin{cases}
2 a_2 + a_4^2=0,&(1)\\
2 a_1^2 a_2^3 + a_1^3 a_2 a_3 + 2 a_2 a_3^6 + a_1^4 a_4 + a_3^6 a_4^2 + \\
\quad a_2^4 a_4^3 + a_2 a_3^4 a_4^3 + a_1 a_3^3 a_4^4 + 2 a_2^3 a_4^5 + a_1^2 a_4^6+2 a_4^9 a_2+a_4^{11}=0.&(2)
\end{cases}
\]
Note that the last two terms in $(2)$ do not appear in the case $r=3$. When $r>3$ these terms become, by $(1)$,
\[ 2 a_4^9 a_2+a_4^{11} = 3 a_4^{11}=0. \]
Hence we can omit the last two terms in $(2)$ in all cases.

Now consider the power $f(x)^{3m+1}$. We are interested in terms of the form $x^{i(q-1)}=x^{i(6m+2)}$. Now $\deg{(f(x)^{3m+1})}=6(3m+1)=3(6m+2)$, so we are interested in the coefficients of $x^{6m+2}$, $x^{2(6m+2)}$, $x^{3(6m+2)}$. By \eref{eq:multinomial3} this amounts to solving, for each $i \in \{ 1,2,3 \}$, the system
\begin{equation}\label{eq:multinomial_system}
\begin{cases}
k_1+k_2+k_3+k_4+k_6=3m+1, &(a) \\
k_1+2k_2+3k_3+4k_4+6k_6 =2i(3m+1). &(b)
\end{cases}
\end{equation}
For $i=3$ it is immediate that the only solution is
\begin{equation*}
 \begin{array}{c|c|c|c|c}
  k_1 & k_2 &k_3 &k_4& k_6 \\
  \hline
	0&0&0&0&3m+1
 \end{array}.
\end{equation*}
For $i=2$ there are many solutions, but we only solve for those for which the multinomial coefficient $\left( \begin{smallmatrix} 3m+1 \\ k_1,k_2,k_3,k_4,k_6 \end{smallmatrix} \right)$ is nonzero mod 3. To apply \tref{thm:lucas} we need expressions for the 3-adic expansions of $3m+1,k_1,k_2,k_3,k_4,k_6$. Now by \eref{eq:3adic_m} we have
\[ 3m+1=1+1\cdot 3+\cdots+1\cdot 3^{r-2}+1\cdot3^{r-1}, \]
and we will denote 3-adic expansions of $k_1,k_2,k_3,k_4,k_6$ by
\begin{gather*}
k_1 = b_{10} + b_{11} \cdot 3+\cdots+b_{1(r-2)}\cdot 3^{r-2}+b_{1(r-1)} \cdot 3^{r-1} \\
	\vdots \\
k_6 = b_{60} + b_{61} \cdot 3+\cdots+b_{6(r-2)}\cdot 3^{r-2}+b_{6(r-1)} \cdot 3^{r-1}
\end{gather*}
Then by \tref{thm:lucas} the multinomial coefficient $\left( \begin{smallmatrix} 3m+1 \\ k_1,k_2,k_3,k_4,k_6 \end{smallmatrix} \right)$ is nonzero mod 3 if and only if
\begin{equation}\label{eq:3adic_ks}
b_{1j}+b_{2j}+b_{3j}+b_{4j}+b_{6j}=1 \text{ for all } 0 \leq j \leq r-1.
\end{equation}
Then \eref{eq:3adic_ks} implies $(\ref{eq:multinomial_system}, a)$. Rewriting $(\ref{eq:multinomial_system}, b)$ with the 3-adic expansions we have
\begin{multline}\label{eq:3adic_multinomial}
(b_{10}+2b_{20}+3b_{30}+4b_{40}+6b_{60})+(b_{11}+2b_{21}+3b_{31}+4b_{41}+6b_{61})\cdot 3+ \cdots +\\
 (b_{1(r-1)}+2b_{2(r-1)}+3b_{3(r-1)}+4b_{4(r-1)}+6b_{6(r-1)})3^{r-1} \\
= 1+2\cdot 3+\cdots+2\cdot 3^{r-1}+3^r
\end{multline}
By \eref{eq:3adic_ks} and \eref{eq:3adic_multinomial} we must have $b_{2j}=0$ for all $j$. For clearly $b_{20}\neq 1$, and if $b_{2j}=1$ for some $1 \leq j \leq r-2$ then it follows that $b_{2(j+1)}=1$. We must then conclude that $b_{2j}=b_{2(j+1)}=\cdots=b_{2(r-1)}=1$, but then the \emph{LHS} of \eref{eq:3adic_multinomial} is too small. So we must have $b_{2j}=0$ for all $j$. By a similar argument we have $b_{1j}=0$ for all $j$, because if $b_{1j}=1$ for some $0 \leq j \leq r-2$ then we must have $b_{2(j+1)}=1$. Hence \eref{eq:3adic_multinomial} reduces to
\begin{multline}\label{eq:3adic_multinomial2}
(3b_{30}+4b_{40}+6b_{60})+ (3b_{31}+4b_{41}+6b_{61}) \cdot 3 + \cdots +\\
(3b_{3(r-1)}+4b_{4(r-1)}+6b_{6(r-1)})\cdot 3^{r-1} \\
= 1+2\cdot 3+\cdots+2\cdot 3^{r-1} +3^r
\end{multline}
It is clear that the only possible solution satisfying \eref{eq:3adic_ks} and \eref{eq:3adic_multinomial2} is $b_{40}=b_{41}=\cdots=b_{4(r-1)}=1$ with all other terms zero. Hence $k_4=1+3+\cdots+3^{r-1}=3m+1$ and the solution is given by
\begin{equation*}
 \begin{array}{c|c|c|c|c}
  k_1 & k_2 &k_3 &k_4& k_6 \\
  \hline
	0&0&0&3m+1&0
 \end{array}
\end{equation*}

For $i=1$ a similar (but easier) argument shows that the only solution with nonzero multinomial coefficient is $k_2=1+3+\cdots+3^{r-1}=3m+1$ and the solution is given by 
\begin{equation*}
 \begin{array}{c|c|c|c|c}
  k_1 & k_2 &k_3 &k_4& k_6 \\
  \hline
	0&3m+1&0&0&0
 \end{array}
\end{equation*}
Hence the term in $x^{q-1}$ in the reduction of $f(x)^{3m+1} \mod (x^q-x)$ is given by
\[ \left( \left( \begin{smallmatrix} 3m+1 \\ 0,0,0,0,3m+1 \end{smallmatrix} \right) + \left( \begin{smallmatrix} 3m+1 \\ 0,0,0,3m+1,0 \end{smallmatrix} \right) a_4^{3m+1}  + \left( \begin{smallmatrix} 3m+1 \\ 0,3m+1,0,0,0 \end{smallmatrix} \right) a_2^{3m+1} \right) x^{q-1}. \]
Thus \tref{thm:hermite} requires that
\[ 1 + a_4^{3m+1} + a_2^{3m+1} =0. \quad(3) \]
\end{proof}

The following theorem rewards the lengthy and tricky calculations in the previous lemma by reducing the characterisation problem to the case $a_5 \neq 0$. 
\begin{theorem}\label{thm:a5not0}
If $r>2$ then there are no  \PP s of $\mathbb{F}_{3^r}$ of the form
\[ 
f(x)=x^6+a_4 x^4 + a_3 x^3 + a_2 x^2 + a_1x .
\]
\end{theorem}

\begin{proof}
If $f(x)$ is a \PP\ of $\mathbb{F}_{3^r}$ then by \lref{lem:a5not0equations} we have
\[
\begin{cases}
a_2 = a_4^2, &(1)\\
2 a_1^2 a_2^3 + a_1^3 a_2 a_3 + 2 a_2 a_3^6 + a_1^4 a_4 +  \\
\quad a_3^6 a_4^2 + a_2^4 a_4^3 + a_2 a_3^4 a_4^3 + a_1 a_3^3 a_4^4 + 2 a_2^3 a_4^5 + a_1^2 a_4^6=0,&(2)\\
1 + a_4^{3m+1} + a_2^{3m+1} =0.&(3)
\end{cases}
\]

Substituting $(1)$ into $(3)$ we have
\begin{equation}\label{eq:a4square}
1 + a_4^{3m+1} + a_4^{6m+2} =0.
\end{equation}
This shows that $a_4 \neq 0$, so we have $a_4^{6m+2}=a_4^{q-1}=1$. We may therefore write \eref{eq:a4square} as
\[ a_4^{3m+1}=a_4^{(q-1)/2}=1. \]
By \eref{eq:squares}, $a_4$ is a nonzero square in $\mathbb{F}_{3^r}$. 

Substituting $(1)$ into $(2)$ we have
\begin{align*}
0 &= a_1^4 a_4 + a_1^3 a_3 a_4^2 + a_1 a_3^3 a_4^4 + a_3^4 a_4^5 \\
	&=a_4(a_1^4 + a_1^3 a_3 a_4 + a_1 a_3^3 a_4^3 + a_3^4 a_4^4) \\
	&=a_4(a_1+a_3a_4)^4.
\end{align*}
Since $a_4 \neq 0$ we have $a_1=2a_3a_4$. But the polynomial
\[ x^6 + a_4 x^4 + a_3 x^3 + a_4^2 x^2 + 2a_3a_4 x\]
has roots at 0 and $a_4^{1/2} \neq 0$, thus failing to be injective.
\end{proof}

Since we know that $a_5 \neq 0$ we will therefore let $a_4=0$ by \lref{lem:a5ora4=0}. We first give a full characterisation for the case $r=3$.
\begin{theorem}
The complete list of \PP s of $\mathbb{F}_{3^3}$ of the form 
\[
f(x)=x^6+a_5 x^5 +a_3x^3+a_2x^2+a_1x
\]
with $a_5 \neq 0$ is given by
\[ x^6+a x^5 + 2 a^4 x^2 \quad (a\neq0). \]
\end{theorem}

\begin{proof}
If $f(x)$ is a \PP\ of $\mathbb{F}_{3^3}$, then by \tref{thm:hermite} the reductions of $f(x)^5$, $f(x)^7$, $f(x)^8$ and $f(x)^{13}$ modulo $x^{27}-x$ must have degree $\leq 25$. These require, respectively,
\[
\begin{cases}
a_5^4+a_3 a_5+a_2=0, &(1)\\
a_2^3 a_5^4+ a_2 a_3^3 a_5^3+ a_1^3 a_5+ a_2^4  =0, &(2)\\
2 a_1^2 a_2^3 + a_1^3 a_2 a_3 + 2 a_2 a_3^6 + a_2^3 a_3^3 a_5 + \\ \quad 2 a_3^7 a_5 +
 a_1^3 a_2 a_5^3 + a_2^4 a_3 a_5^3 +2 a_1^2 a_3^3 a_5^3 + a_1^3 a_3 a_5^4=0, &(3) \\
1 + a_2^{13} + a_1^3 a_2^9 a_5 + a_1^9 a_2 a_5^3 + a_1 a_2^3 a_5^9=0.&(4)
\end{cases}
\]
Applying (1) to (2) we have 
\begin{align*}
0 &=a_5(2 a_1^3 + a_3^3 a_5^6 + 2 a_3 a_5^{12}\\
	&= a_5 (2 a_1 + a_3 a_5^2 + 2 a_3^9 a_5^4)^3.
\end{align*}
Hence 
\begin{equation}\label{eq:a1_F27}
a_1=a_3 a_5^2 + 2 a_3^9 a_5^4.
\end{equation}

If $a_3=0$ then by \eref{eq:a1_F27} we have $a_1=0$. By checking the remaining powers in Hermite's criterion, the resulting family
\[ x^6+a x^5 + 2 a^4 x^2 \quad (a\neq0) \]
are shown to be \PP s.

Now suppose that $a_3\neq0$. Applying (1) and \eref{eq:a1_F27} to (3) we have
\begin{align*}
0 &= 2 a_3^3 a_5^{13} + a_3^{10} a_5^{18} + 2 a_3 a_5^{19} + a_3^{18} a_5^{20} \\
	&=a_3 a_5^{13} (2 a_3^2 + a_3^9 a_5^5 + 2 a_5^6 + a_3^{17} a_5^7).
\end{align*}
Dividing by $a_3 a_5^{13} \neq 0$ and letting $a_3 = \eta a_5^3$, this becomes 
\[  \eta^{17} + \eta^9 + 2 \eta^2 +2. \]
Multiplying this equation by $\eta^9$ and letting $\eta^{26}=1$ and $\eta^{11}=\eta^{52} \cdot \eta^{11}=\eta^{63}$, we have
\begin{align*}
0 &=  \eta^{63} + 2 \eta^{18} + \eta^9 +2 \\
	&=(\eta +2)^{18} (\eta^2 +1)^9 (\eta^3 + 2 \eta^2+2 \eta +2)^9
\end{align*}
Now the root $\eta =1$ can be ignored since the polynomial 
\[ x^6+ a_5 x^5 + a_5^3 x^3 +a_5^4 x^2 \]
has roots at 0 and $-a_5 \neq 0$, and $\eta^2=-1$ is impossible since $\mathbb{F}_{3^3}$ is a degree 3 extension of $\mathbb{F}_3$. So we must have
\[ \eta^3 = \eta^2+\eta +1. \]
Substituting $a_1=a_5^5(\eta + 2 \eta^9),a_2=2 a_5^4(1+ \eta)$ and $a_3=a_5^3 \eta$ into (4) and simplifying gives 
\[ 2 \eta (\eta^3+ 2 \eta+2) (\eta^3 + \eta^2+2) (\eta^3 + \eta^2 + 2 \eta+1) (\eta^3 + 2 \eta^2+1) =0. \]
But using $\eta^3 = \eta^2+\eta +1$ this simplifies to
\[ 2 \eta^4 (\eta +2)^2 (\eta^2+1)=0, \]
in contradiction to $\eta \not\in \{ 0, 1 \}$ and $\eta^2 \neq -1$. 
\end{proof}

Finally, we show that there are no \PP s of $\mathbb{F}_{3^r}$ when $r>3$. 
\begin{theorem}\label{thm:nopps3k}
If $r>3$ and $a_5 \neq 0$ then there are no \PP s of $\mathbb{F}_{3^r}$ of the form
\[ 
f(x)=x^6+a_5 x^5 + a_3 x^3 + a_2 x^2 + a_1x.
\]
\end{theorem}

\begin{proof}
Suppose $f(x)$ is a \PP\ of $\mathbb{F}_{3^r}$ and $r>3$. By \lref{lem:m_inequality3} we have $1 \leq m \leq q-8$ and $m \equiv 1 \mod 3$, so by \tref{thm:hermite} the reductions of $f(x)^{m+1}$ and $f(x)^{m+3}$ modulo $x^q-x$ have degree $\leq q-2$. By similar calculations to \lref{lem:a5not0equations} we determine that these conditions require
\[
\begin{cases}
a_2 = 2a_3 a_5 + 2a_5^4, &(1)\\
a_2^4 +a_1^3 a_5 + a_2 a_3^3 a_5^3 + a_2^3 a_5^4 + a_2 a_5^{12} + 2a_5^{16}=0. &(2)
\end{cases}
\]
Substituting (1) into (2) we have
\begin{align*}
0 &= 2 a_1^3 a_5 + a_3^3 a_5^7 + 2 a_5^{16}\\
	&= 2a_5 (a_1 +2a_3 a_5^2+a_5^5)^3.
\end{align*}
Since $a_5 \neq 0$ we must have $a_1=a_3a_5^2+2a_5^5$. But the polynomial 
\[ x^6+a_5 x^5 + a_3 x^3 + (2a_3 a_5 + 2a_5^4) x^2 + (a_3a_5^2+2a_5^5) x\]
has roots 0 and $-a_5 \neq 0$, thus failing to be injective, so is not a \PP.
\end{proof}

\section{Normalised PPs of Degree 6}
In Theorems \ref{thm:nopps>11} - \ref{thm:nopps3k} we derived the complete classification of degree 6 \PP s of \fq\ up to transformations of the form 
\[ c f(x+b) +d, \text{ where } b,c,d \in \mfq, c\neq 0. \]
These are listed in Table \ref{tab:deg6_extrareduced}.

\begin{table}[h!]
\setlength{\abovecaptionskip}{1pt}
\begin{my_enumerate}
\item $x^6 \pm 2x$, $q=11$.
\item $x^6 \pm 4x$, $q=11$.
\item $x^6 \pm a^2x^3+ax^2 \pm 5x$, $a$ a nonzero square, $q=11$.
\item $x^6 \pm 4 a^2 x^3+ax^2 \pm 4x$, $a$ not square, $q=11$.
\item $x^6+a^2x^4+a^7bx^3+a^4x^2+a(2b+1)x$, $a \neq 0, b\in \{0,1,2^{1/2},1+2^{1/2} \}$, $q=3^2$.
\item $x^6+ax^5+a^3x^3+2a^4x^2+2a^5x$, $a \neq 0$, $q=3^2$.
\item $x^6+ax^5+\varphi a^3x^3+2\varphi a^4x^2+2^{1/2}a^5x$, $a\neq 0,\varphi= \pm (1-2^{1/2})$, $q=3^2$.
\item $x^6+ax^5+2a^3x^3+a^4x^2+(2+2^{1/2})a^5x$, $a \neq 0$, $q=3^2$.
\item $x^6+ax^5+2a^4x^2$, $a \neq 0$, $q=3^3$.
\end{my_enumerate}
\caption{\emph{Classification of degree 6 \PP s of \fq\ ($q$ odd) up to linear transformations.}}
\label{tab:deg6_extrareduced}
\end{table}
\pagebreak

We remark that this list agrees with the original list in \cite{dickson1897II}, except that Dickson fails to specify that $a$ is not allowed to be zero in (3) and (6-9). Additionally, the family of \PP s given in (5) has a cleaner parametrisation than in the original list.

We note, however, that this is \emph{not} the complete list of normalised degree 6 \PP s in the sense of \dref{def:normalised}, because in the cases $q=3^2$ and $q=3^3$ we used a second linear transformation to ensure that the coefficient of $x^{5}$ or $x^4$ was zero. So, for example, the polynomial
\[f(x) = x^6+x^5+x^4+x^2 \in \mathbb{F}_{3^2}[x] \]
is a normalised \PP\ of $\mathbb{F}_{3^2}$ not appearing in the above list. Similarly, the following polynomial is a normalised \PP\ of $\mathbb{F}_{3^3}$ not in the list:
\[ g(x)=x^6+x^5+2x^4 \in \mathbb{F}_{3^3}[x]. \]
So, although the above list completely classifies degree 6 \PP s up to linear transformations, it is incorrect to call it a complete list of \emph{normalised} \PP s. This is the claim made by Dickson in \cite{dickson1897II} when he said his list was a \emph{complete list of reduced quantics}. Although the distinction is minor, we suggest this ambiguity has caused confusion and is the reason that Dickson's characterisation has been questioned.

We now convert the list (1-9) into the complete list of normalised \PP s of degree 6, which we feel is necessary for the sake of consistency and avoiding future confusion. We will then be able to insert the degree 6 classification unambiguously into the complete list of normalised (in the sense of the globally accepted definition) \PP s of degree up to 6. Recall from Section \ref{sec:3k} that in the case $p=3$ and $a_5 \neq 0$ we used a linear transformation to remove the $x^4$ term (see \lref{lem:a5ora4=0}). To recover the list of normalised \PP s represented by (6-9) we must apply transformations of the form
\[ f(x) = g(x+b) +c, \]
where $b$ is arbitrary and $c$ is chosen so that the resulting polynomial $f(x)$ satisfies $f(0)=0$. Applying these transformations to the polynomials (6-9), and simplifying, we obtain the complete list of normalised \PP s of degree 6 in Table \ref{tab:deg6_normalised}. See Appendix \ref{appendix:list} for the complete list of normalised \PP s of degree $\leq 6$.
\begin{table}[h!]
\begin{my_enumerate}
\item $x^6 \pm 2x$, $q=11$.
\item $x^6 \pm 4x$, $q=11$.
\item $x^6 \pm a^2x^3+ax^2 \pm 5x$, $a$ a nonzero square, $q=11$.
\item $x^6 \pm 4 a^2 x^3+ax^2 \pm 4x$, $a$ not square, $q=11$.
\item $x^6+a^2x^4+a^7bx^3+a^4x^2+a(2b+1)x$, $a \neq 0, b\in \{0,1,2^{1/2},1+2^{1/2} \}$, $q=3^2$.
\item $x^6 + a x^5 +2 a b x^4 + (a^3 + a b^2 + 2 b^3) x^3 + (2 a^4 + a b^3) x^2+ (2 a^5 + a^4 b + 2 a b^4) x$, $a \neq 0$, $b$ arbitrary, $q=3^2$.
\item $ x^6+a x^5 +2 a b x^4 + (a b^2 + 2 b^3 + a^3 \varphi)x^3 + 
  (a b^3 + 2a^4 \varphi)x^2 + (2^{1/2} a^5 + 2 a b^4 + a^4 b \varphi)x$, $a\neq 0$, $b$ arbitrary, $\varphi= \pm (1-2^{1/2})$, $q=3^2$.
\item $x^6 + a x^5 + 2 a b x^4 + (2 a^3 + a b^2 + 2 b^3) x^3 + (a^4 + a b^3) x^2 +(2 a^5 + 2^{1/2} a^5 + 2 a^4 b + 2 a b^4) x$, $a \neq 0$, $b$ arbitrary, $q=3^2$.
\item $x^6+  a x^5 + 2 a b x^4 +(a b^2 + 2 b^3) x^3 + (2 a^4 + a b^3) x^2 + (a^4 b + 2 a b^4) x$, $a \neq 0$, $b$ arbitrary, $q=3^3$.
\end{my_enumerate}
\caption{\emph{Complete list of normalised degree 6 \PP s of \fq\ ($q$ odd).}}
\label{tab:deg6_normalised}
\end{table}

\chapter{Orthomorphism Polynomials}
\section{Orthomorphism Polynomials of Finite Fields}
We begin by defining orthomorphisms for general finite groups $G$.
\begin{newdef}
Let $G$ be a finite group. Then an \emph{orthomorphism} of $G$ is a permutation $\Phi$ of $G$ such that the map $c \mapsto c^{-1} \Phi(c)$ is also a permutation of $G$.
\end{newdef}

Orthomorphisms have also been referred to as \emph{orthogonal mappings}. There are numerous reasons to be interested in orthomorphisms, for example for the construction of orthogonal Latin squares.

A closely related concept is that of a \emph{complete mapping} of $G$, which is a permutation $\Phi$ such that the map $c \mapsto c \Phi(c)$ is a permutation of $G$. Then $\Phi$ is an orthomorphism of $G$ if and only if the map $c \mapsto c^{-1} \Phi(c)$ is a complete mapping of $G$ and a complete mapping of $G$ if and only if the map $c \mapsto c \Phi(c)$ is an orthomorphism of $G$.

In this paper we only consider orthomorphisms of the additive group $\mfq^{+}$ of a finite field; the interested reader may refer to \cite{evans1992} for orthomorphisms of general groups. Note that by \lref{lem:unique_poly} we may assume that an orthomorphism of $\mfq^{+}$ is a polynomial $f \in \mfqx$. We will call such a polynomial an \emph{orthomorphism polynomial}. It is clear that $f$ is an orthomorphism polynomial of \fq\ if and only if $f(x)$ and $f(x)-x$ are both \PP s of \fq. 

We begin by stating a fundamental result on the degree of an orthomorphism polynomial of \fq. We already know (by \cref{cor:max_degree}) that the reduction modulo $x^q-x$ of a \PP\ of \fq\ has degree at most $q-2$. In fact for orthomorphism polynomials we have the following stronger bound. 
\begin{theorem}
If $q>2$ and $f(x)$ is an orthomorphism polynomial of \fq\ then the reduction of $f$ modulo $x^q-x$ has degree at most $q-3$.
\end{theorem}
The above theorem was proved by Niederreiter and Robinson \cite{nieder1982} for odd $q$, and by Wan \cite{wan1986} for even $q$. We refer the reader to \cite{evans1992} for its complete proof.

The following trivial lemma gives us a concept analogous to normalised a permutation polynomials.
\begin{lemma}\label{lem:normalised_ortho}
If $f \in \mfqx$ is an orthomorphism polynomial of \fq\ then so is
\[ g(x)=f(x+b)+d, \text{ where } b,d \in \mfq. \]
\end{lemma}
By suitable choices of $b$ and $d$ we can ensure that the resulting polynomial $g(x)$ satisfies $g(0)=0$, and when the degree $n$ of $f$ is not divisible by the characteristic of \fq, the coefficient of $x^{n-1}$ is zero. 

\begin{remark}
Unlike with permutation polynomials, it is \emph{not} true that $c f(x)$ is also an orthomorphism polynomial for any $c \neq 0$. For example, $2x$ is always an orthomorphism polynomial of \fq\ but $x$ is not. 
\end{remark}

\section{Degree 6 Orthomorphism Polynomials}
Orthomorphism polynomials of degree up to 5 were classified by Niederreiter and Robinson \cite{nieder1982} in 1982. In fact they actually classified complete mapping polynomials, but it is then a simple matter to determine the orthomorphism polynomials since $f(x)$ is a complete mapping polynomial if and only if $f(x)+x$ is an orthomorphism polynomial. In the same paper the authors also resolved the degree 6 case when $\gcd(6,q)=1$. The reader may find the list of orthomorphism polynomials from these cases in \cite{evans1992}. Using the characterisation of degree 6 \PP s from \chref{chap:deg6} we now proceed to classify all degree 6 orthomorphism polynomials of fields of characteristic 3.

Let $f(x)$ be an orthomorphism polynomial of $\mathbb{F}_{3^r}$, where $r \geq 2$. Then by \lref{lem:normalised_ortho} the polynomial $g(x)=f(x+b)+d$ is also an orthomorphism of $\mathbb{F}_{3^r}$. By choosing $b$ and $d$ suitably we can ensure that $g(0)=0$ and if the coefficient of $x^5$ is nonzero then the coefficient of $x^4$ is zero. If we can classify orthomorphisms with these properties then we have a complete classification up to linear transformations. 

\noindent Let $f \in \mathbb{F}_{3^r}[x]$ be a degree 6 polynomial and define the following properties:
\begin{description} [leftmargin=15mm,labelindent=5mm,style=sameline,itemsep=1pt,font=\normalfont]
\item[(\emph{P1})] $f(0)=0$.
\item[(\emph{P2})] The coefficient of $x^5$ is zero.
\item[(\emph{P3})] The coefficient of $x^5$ is nonzero and the coefficient of $x^4$ is zero.
\end{description}
Also, recall that in this project the symbol $2^{1/2}$ always represents \emph{either} root of the equation $x^2+1=0$ in $\mathbb{F}_{3^2}$. It will henceforth be necessary to distinguish between the two solutions, so we let $2^{1/2} \in \{ \pm i \}$, where $i$ is a solution fixed throughout. 

We first classify the orthomorphism polynomials of $\mathbb{F}_{3^2}$ satisfying (\emph{P1}) and (\emph{P2}).
\begin{theorem}\label{thm:orthos_F9_a5zero}
The complete list of degree 6 orthomorphism polynomials $f(x)$ of $\mathbb{F}_{3^2}$ satisfying (P1) and (P2) is given by 
\begin{gather*}
c x^6+c^7 x^4 + c^5 x^2 + 2x \quad (c \neq 0), \\
c x^6+c^7 x^4 - 2^{1/2} c^2 x^3+ c^5 x^2 + (2 + 2^{1/2})x \quad (c \neq 0).
\end{gather*}
\end{theorem}

\begin{proof}
Recall that Table \ref{tab:deg6_extrareduced} (5) is the complete classification of monic \PP s of $\mathbb{F}_{3^2}$ satisfying (\emph{P1}) and (\emph{P2}). Multiplying this family by an arbitrary constant $c\neq0$ (to remove the monoticity restriction), we recover the complete list of \PP s satisfying (\emph{P1}) and (\emph{P2}):
\begin{equation}\label{eq:class_a5zero_F9}
\begin{gathered}
cx^6+a^2cx^4+a^7bcx^3+a^4cx^2+ac(2b+1)x,\\
a,c \neq 0, b\in \{0,1,\pm i,1 \pm i\}.
\end{gathered}
\end{equation}
If $f(x)$ is an orthomorphism polynomial of $\mathbb{F}_{3^2}$ then it is necessarily of the above form; that is,
\begin{equation}\label{eq:f(x)_a5zero_F9}
f(x)=cx^6+a^2cx^4+a^7bcx^3+a^4cx^2+ac(2b+1)x
\end{equation}
for some $a,c \neq 0, b\in \{0,1,\pm i,1 \pm i\}$. Hence,
\begin{equation}\label{eq:f(x)-x_a5zero_F9}
f(x)-x=cx^6+a^2cx^4+a^7bcx^3+a^4cx^2+(ac(2b+1)-1)x.
\end{equation}
Note that this is again a polynomial satisfying (\emph{P1}) and (\emph{P2}), so $f$ is an orthomorphism polynomial if and only if $f(x)-x$ is of the form \eref{eq:class_a5zero_F9}. That is, there must exist $A,C \neq 0, B\in \{0,1,\pm i,1 \pm i\}$ such that 
\begin{equation}\label{eq:f(x)-x_general_a5zero_F9}
f(x)-x = Cx^6+A^2C x^4+A^7BCx^3+A^4Cx^2+AC(2B+1)x.
\end{equation}
We proceed to equate the coefficients of \eref{eq:f(x)-x_a5zero_F9} and \eref{eq:f(x)-x_general_a5zero_F9}. Since clearly it is impossible that $A=a,B=b,C=c$, we have from the coefficients of $x^6,x^4,x^3$ and $x^2$ that 
\[ A=-a,B=-b,C=c. \] 
Hence $b \in \{ 0, \pm i \}$, the largest subset of $\{0,1,\pm i,1 \pm i\}$ closed under negation. The coefficient of $x$ then requires
\[ ac(2b+1)-1=-ac(-2b+1).\]
Rearranging, we find that $a=-c^{-1}$, and substituting this and $b \in \{0,\pm i\}$ into \eref{eq:f(x)_a5zero_F9} and simplifying we have precisely the following orthomorphism polynomials:
\[
c x^6+c^7 x^4 - b c^2 x^3+ c^5 x^2 + (b+2)x \quad (c \neq 0, b \in \{ 0, \pm i \}).
\]
\end{proof}

Now we consider the case where (\emph{P1}) and (\emph{P3}) are satisfied.
\begin{theorem}
The complete list of degree 6 orthomorphism polynomials $f(x)$ of $\mathbb{F}_{3^2}$ satisfying (P1) and (P3) is given by
\[ a^5 x^6+ 2^{1/2} a^4 x^5+ 2^{1/2} a^2 x^3+a x^2 + 2(1+2^{1/2}) x \quad (a \neq 0).\]
\end{theorem}

\begin{proof}
From Table \ref{tab:deg6_extrareduced} (6-8) the complete list of \PP s of $\mathbb{F}_{3^2}$ satisfying (\emph{P1}) and (\emph{P3}) is given by
\begin{gather}
cx^6+acx^5+a^3cx^3+2a^4cx^2+2a^5cx, \quad(a,c \neq 0) \label{class_a4zero_1}\\ 
cx^6+acx^5+\varphi a^3cx^3+2\varphi a^4cx^2 + i a^5cx, \quad(a,c\neq 0,\varphi = \pm(1-i)) \label{class_a4zero_2} \\
cx^6+acx^5+\varphi a^3cx^3+2\varphi a^4cx^2 - i a^5cx, \quad(a,c\neq 0,\varphi = \pm(1+i)) \label{class_a4zero_3} \\
cx^6+acx^5+2a^3cx^3+a^4cx^2+(2 + i)a^5cx, \quad(a,c \neq 0) \label{class_a4zero_4} \\
cx^6+acx^5+2a^3cx^3+a^4cx^2+(2 - i)a^5cx. \quad(a,c \neq 0) \label{class_a4zero_5}
\end{gather}
As before, note that if $f(x)$ is a polynomial satisfying (\emph{P1}) and (\emph{P3}) then $f(x)-x$ also satisfies (\emph{P1}) and (\emph{P3}). Hence, $f$ is an orthomorphism polynomial if and only if $f(x)$ and $f(x)-x$ both appear in the complete list of \PP s above. 

It is not difficult to see that if $f(x)$ is of the form \eref{class_a4zero_1} then it is not possible for $f(x)-x$ to be of any of the forms \eref{class_a4zero_1}-\eref{class_a4zero_5}. So there are no orthomorphism polynomials of the form \eref{class_a4zero_1}. Similar reasoning shows that the only possibility for $f(x)$ and $f(x)-x$ to both be on the list is if one is of the form \eref{class_a4zero_4} and the other is of the form \eref{class_a4zero_5}. First consider the case where $f(x)$ is of the form \eref{class_a4zero_4}. Let
\begin{equation}\label{eq:f(x)_a4zero_F9}
f(x)=cx^6+acx^5+2a^3cx^3+a^4cx^2+(2 + i)a^5cx,
\end{equation}
where $a,c\neq 0$. Then 
\begin{equation}\label{eq:f(x)-x_a4zero_F9}
f(x)-x=cx^6+acx^5+2a^3cx^3+a^4cx^2+((2 + i)a^5c-1)x.
\end{equation}
For $f(x)-x$ to be of the form \eref{class_a4zero_5} we there must exist $A,C\neq 0$ such that 
\begin{equation}\label{eq:f(x)-x_general_a4zero_F9}
f(x)-x=Cx^6+ACx^5+2A^3Cx^3+A^4Cx^2+(2 - i)A^5Cx.
\end{equation}
Equating coefficients of \eref{eq:f(x)-x_a4zero_F9} and \eref{eq:f(x)-x_general_a4zero_F9} we conclude that $C=c,A=a$ and that 
\[ (2 + i)a^5c-1 =(2 - i)a^5c. \]
Rearranging, we have $c=a^3 i$. Substituting this into \eref{eq:f(x)_a4zero_F9} we have
\[ f(x)=a^3 i x^6+a^4 i x^5 + 2a^6 i x^3 + a^7 i x^2 + 2(1+i) x\]
Replacing $a$ with $a^{-1}i$ we have
\[ f(x)=a^5 x^6+a^4 i x^5+a^2 i x^3+a x^2 + 2(1+i) x.\]
In a similar fashion we determine that $f(x)$ is of the form \eref{class_a4zero_5} and $f(x)-x$ is of the form \eref{class_a4zero_4} if and only if
\[ f(x)=a^5 x^6-a^4 i x^5-a^2 i x^3+a x^2 + 2(1-i) x.\]
\end{proof}

Finally we show that there are no degree 6 orthomorphism polynomials of $\mathbb{F}_{3^r}$ for any $r>2$. 
\begin{theorem}\label{thm:no_orthos_3r}
There are no degree 6 orthomorphism polynomials of $\mathbb{F}_{3^r}$ for any $r>2$. 
\end{theorem}

\begin{proof}
By Theorems \ref{thm:a5not0} and \ref{thm:nopps3k} there are no \PP s of $\mathbb{F}_{3^r}$ for any $r>3$, so certainly there are no orthomorphism polynomials. 

Let $r=3$ and let $f(x)$ be a degree 6 orthomorphism polynomial of $\mathbb{F}_{3^3}$. By a linear transformation we may assume that $f$ satisfies (\emph{P1}) and either (\emph{P2}) or (\emph{P3}). By Table \ref{tab:deg6_extrareduced} (9) we then have
\begin{equation}\label{eq:f_F27}
 f(x)= cx^6+acx^5+2a^4cx^2,
\end{equation}
for some $a,c \neq 0$. But clearly $f(x)-x$ cannot also be of the form \eref{eq:f_F27}, so $f(x)-x$ is not a \PP, a contradiction.
\end{proof}

By Theorems \ref{thm:orthos_F9_a5zero} - \ref{thm:no_orthos_3r} there are degree 6 orthomorphism polynomials of $\mathbb{F}_{3^r}$ if and only if $r=2$. The following is the classification of degree 6 orthomorphism polynomials of $\mathbb{F}_{3^2}$:
\begin{gather}
 a x^6+a^7 x^4 + a^5 x^2 +2 x, \quad (a \neq 0) \label{eq:first_ortho_F9} \\
 a x^6+a^7 x^4 - 2^{1/2} a^2 x^3+ a^5 x^2 + (2 + 2^{1/2})x, \quad (a \neq 0) \label{eq:second_ortho_F9} \\
 a^5 x^6+ 2^{1/2} a^4 x^5+ 2^{1/2} a^2 x^3+a x^2 + 2(1+2^{1/2}) x. \quad (a \neq 0) \label{eq:third_ortho_F9}
\end{gather}
Every degree 6 orthomorphism polynomial of $\mathbb{F}_{3^2}$ is of one of the forms
\begin{itemize}
 \item $f(x)+d$, where $f(x)$ is of the form \eref{eq:first_ortho_F9} or \eref{eq:second_ortho_F9} and $d \in \mathbb{F}_{3^2}$, or
 \item $g(x+b)+d$, where $g(x)$ is of the form \eref{eq:third_ortho_F9} and $b,d \in \mathbb{F}_{3^2}$.
\end{itemize}

\appendix
\chapter{List of Normalised PPs}\label{appendix:list}
With the exception of degree 6 polynomials in even characteristic, the following table is the complete list of normalised permutation polynomials of degree $\leq 6$. The reader is referred to the recent paper \cite{li2010} by Li \emph{et al.} for the classification of \PP s of degree 6 and 7 over fields of even characteristic. 
\begin{table}[h!]
\begin{tabularx}{\textwidth}{ | X| l | }
\hline
Normalised \PP & $q$ \\
\hline
$x$ & any $q$ \\
\hline
$x^2$ & $q \equiv 0 \mod 2$ \\
\hline
$x^3$ & $q \not \equiv 1 \mod3$ \\
$x^3-ax$, $a$ not square & $q \equiv 0 \mod3$ \\
\hline
$x^4 \pm 3x$ & $q =7$ \\
$x^4+a_1x^2+a_2x$, if its only root in \fq\ is 0 & $q \equiv 0 \mod2$ \\
\hline
$x^5$ & $q \not\equiv 1 \mod5$ \\
$x^5-ax$, $a$ not a fourth power & $q \equiv 0 \mod5$ \\
$x^5 + 2^{1/2} x$ & $q=9$ \\
$x^5 \pm 2x^2$ & $q=7$ \\
$x^5+ax^3 \pm x^2 +3a^2 x$, $a$ not a square & $q=7$ \\
$x^5 + ax^3 + 5^{-1} a^2x$, $a$ arbitrary & $q \equiv 2,3 \mod5$ \\
$x^5 + ax^3 + 3a^2x$, $a$ not square & $q=13$ \\
$x^5 - 2ax^3 +a^2x$, $a$ not square & $q \equiv 0 \mod5$ \\
\hline
$x^6 \pm 2x$ &$q=11$\\
$x^6 \pm 4x$ & $q=11$\\
$x^6 \pm a^2x^3+ax^2 \pm 5x$, $a$ a nonzero square & $q=11$\\
$x^6 \pm 4 a^2 x^3+ax^2 \pm 4x$, $a$ not square & $q=11$\\
$x^6+a^2x^4+a^7bx^3+a^4x^2+a(2b+1)x$, &$q=3^2$\\
\quad $a \neq 0, b\in \{0,1,2^{1/2},1+2^{1/2} \}$ & \\
$x^6 + a x^5 +2 a b x^4 + (a^3 + a b^2 + 2 b^3) x^3 + (2 a^4 + a b^3) x^2+$ &$q=3^2$ \\
\quad $(2 a^5 + a^4 b + 2 a b^4) x, a \neq 0$, $b$ arbitrary & \\
$ x^6+a x^5 +2 a b x^4 + (a b^2 + 2 b^3 + a^3 \varphi)x^3 + (a b^3 + 2a^4 \varphi)x^2 + $ &$q=3^2$\\
\quad $(2^{1/2} a^5 + 2 a b^4 + a^4 b \varphi)x, a\neq 0$, $b$ arbitrary, &\\
\quad  $\varphi= \pm (1-2^{1/2})$ & \\
$x^6 + a x^5 + 2 a b x^4 + (2 a^3 + a b^2 + 2 b^3) x^3 + (a^4 + a b^3) x^2 +$ &$q=3^2$\\
\quad $(2 a^5 + 2^{1/2} a^5 + 2 a^4 b + 2 a b^4) x, a \neq 0$, $b$ arbitrary & \\
$x^6+  a x^5 + 2 a b x^4 +(a b^2 + 2 b^3) x^3 + (2 a^4 + a b^3) x^2 + $ &$q=3^3$\\
\quad $(a^4 b + 2 a b^4) x, a \neq 0$, $b$ arbitrary & \\
\hline
\end{tabularx}
\caption{\emph{List of Normalised Permutation Polynomials.}}
\label{table:normalised}
\vspace{2mm}
Note that in this table $2^{1/2}$ always occurs as a symbol for \emph{either} root of the polynomial $x^2-2$ in $\mathbb{F}_{3^2}$.
\end{table}

\bibliographystyle{plain}	
\bibliography{permutationbib}
\end{document}